\documentclass[preprint,12pt]{elsarticle}
\usepackage{amsthm}
\usepackage[sumlimits]{amsmath}
\usepackage{amssymb}
\usepackage{MnSymbol}
\usepackage{color}
\usepackage{amscd,lineno}
\usepackage{graphicx}
\usepackage[margin=2.5cm]{geometry}

\DeclareMathOperator{\diff}{d}

\DeclareMathOperator{\Id}{Id}

\newtheorem{theorem}{Theorem}

\newtheorem{remark}[theorem]{Remark}
\def\MM#1{\boldsymbol{#1}}

\newcommand{\pp}[2]{\frac{\partial #1}{\partial #2}} 

\newcommand{\dd}[2]{\frac{\diff#1}{\diff #2}}

\newcommand{\contract}[1]{#1\righthalfcup}

\journal{Journal of Computational Physics}

\begin{document}
\linenumbers
\begin{frontmatter}



\title{A finite element exterior calculus framework for the rotating 
  shallow-water equations}


\author[ic]{C. J. Cotter}
\author[ex]{J. Thuburn}
\address[ic]{Department of Aeronautics, Imperial College London,
London SW7 2AZ, U.K.}
\address[ex]{College of Engineering, Mathematics and Physical Sciences,
University of Exeter, Exeter EX4 4QF, U.K.}

\begin{abstract}
  We describe discretisations of the shallow water equations on the
  sphere using the framework of finite element exterior calculus,
  which are extensions of the mimetic finite difference framework
  presented in Ringler, Thuburn, Klemp, and Skamarock (Journal of
  Computational Physics, 2010). The exterior calculus notation
  provides a guide to which finite element spaces should be used for
  which physical variables, and unifies a number of desirable
  properties.  We present two formulations: a ``primal'' formulation
  in which the finite element spaces are defined on a single mesh, and
  a ``primal-dual'' formulation in which finite element spaces on a
  dual mesh are also used. Both formulations have velocity and layer
  depth as prognostic variables, but the exterior calculus framework
  leads to a conserved diagnostic potential vorticity. In both
  formulations we show how to construct discretisations that have
  mass-consistent (constant potential vorticity stays constant),
  stable and oscillation-free potential vorticity advection.
\end{abstract}

\begin{keyword}
Finite element exterior calculus \sep Potential vorticity
\sep Numerical weather prediction \sep Shallow water equations
\MSC 76M10 \sep 65M60
\end{keyword}

\end{frontmatter}

\section{Introduction}

In a recent paper on horizontal grids for global weather and climate
models, \cite{StTh2012} listed a number of desirable properties that a
numerical discretisation should have, which can be paraphrased as
accurate representation of geostrophic adjustment, mass conservation,
curl-free pressure gradient, energy conserving pressure terms, energy
conserving Coriolis term, steady geostrophic modes, and
absence/control of spurious modes. Of this list as presented here, the
first property could be said to relate to the stability and accuracy
of the discrete Laplacian formed from divergence and gradient
operators, whilst the next five all relate to mimetic properties
(\emph{i.e.} the numerical discretisations exactly represent
differential calculus identities such as $\nabla\times\nabla=0$), and
the last property relates to the kernels of the various discretised
operators (see \cite{Ro+2005,LeRoPo2007,LePo2008} and related papers
by Le Roux and coworkers for extended discussion of these issues in
the context of finite element methods). In the context of the rotating
shallow-water equations on the sphere, which represent the standard
nonlinear framework for investigating horizontal grids for global
models, the C-grid staggering on the latitude-longitude grid combined
with an appropriate choice of reconstruction of the Coriolis term
provides all of these properties, but leaves us with a grid system
with a polar singularity. This, together with a need for models with
variable resolution, has started a quest for alternative grids and
discretisations that satisfy these properties. 

The extension of the C-grid to triangular meshes (and the finite
element analogue, the RT0-P0 discretisation) satisfies the first six
properties and has been popular in both atmosphere and ocean
applications (\cite{WaCa1998,BoRi2005}), however it is now well
understood that the triangular C-grid supports spurious
inertia-gravity mode branches because of the decreased ratio of
velocity degrees of freedom (DOFs) to pressure DOFs relative to
quadrilaterals (from 2:1 to 3:2) \cite{Da2010,Ga2011}. More recently,
a Coriolis reconstruction for the hexagonal C-grid was derived in
\cite{Th08} that provides the mimetic properties described above, and
this was extended to arbitrary orthogonal C-grids (grids in which dual
grid edges that join pressure points intersect the primal grid edges
orthogonally) in \cite{ThRiSkKl2009}. The hexagonal C-grid has an
increased ratio of velocity DOFs to pressure DOFs (from 2:1 to 3:1),
and so does not support spurious inertia-gravity mode branches, but
does have a branch of spurious Rossby modes. This reconstruction can
be used to construct energy and enstrophy conserving C-grid
discretisations for the nonlinear rotating shallow-water equations
using the vector invariant form \cite{RiThKlSk2010}, in which mimetic
properties are used to produce a velocity-pressure formulation in
which the diagnosed potential vorticity is locally conserved in a
shape-preserving advection scheme, and is consistent with the discrete
mass conservation (\emph{i.e.} constant potential vorticity stays
constant in the unforced case).

Two directions remain outstanding from this approach, namely the
relaxation of the orthogonality requirement which constrains cubed
sphere grids so that grid resolution increases much more quickly in
the corners than at the middle of the faces \cite{PuLi2007}, and the
construction of higher-order operators to avoid grid imprinting. Two
recent papers by the authors attempted to address these issues. In
\cite{ThCo2012} a framework was set up to generalise the mimetic
approach of \cite{RiThKlSk2010} to non-orthogonal grids, but the
method of constructing sufficiently high-order operators was not
clear. Meanwhile, in \cite{CoSh2012}, it was shown that mixed finite
element methods in the framework of finite element exterior calculus
(see \cite{ArFaWi2006} for a review) provide the first six properties
listed above, plus sufficient flexibility to adjust the ratio of
velocity DOFs to pressure DOFs to 2:1 to avoid spurious mode
branches. The BDFM1 space on triangles and the RTk hierarchy of spaces
on quadrilaterals were advocated as examples of spaces that satisfy
that ratio.  However, in that paper it was not clear how the extension
to nonlinear shallow-water equations would be made.  In this paper we
address both of these open questions by describing a finite element
exterior calculus framework for the shallow-water equations, which
enables us to write the equations in a very compact form that is
coordinate-free, and reveals the underlying structure behind the
mimetic properties. The goal is to have a numerical discretisation for
the shallow water equations with velocity and layer thickness as
prognostic variables, but with a conserved diagnostic potential
vorticity. We shall discuss two formulations: a primal grid
formulation in which potential vorticity is represented on a
continuous finite element space, and a primal-dual grid formulation
that makes use of the discrete Hodge star operator introduced in
\cite{Hi2001a,Hi2001b} in which potential vorticity is represented on
a discontinuous finite element space. In the latter case,
discontinuous Galerkin or finite volume methods can be used for
locally conservative, bounded, mass-consistent potential vorticity
advection, whilst in the former case we show that streamline-upwind
Petrov-Galerkin methods with discontinuity-capturing schemes can be be
incorporated into the framework to provide conservative, high-order,
stable, non-oscillatory advection of potential vorticity.

Throughout the paper we express our formulations in the language of
differential forms. In \cite{ArFaWi2006,ArFaRa2010} it was shown that
this language provides a unifying structure for a wide range of
different finite element spaces, which provides a coherent framework
for finite element approximation theory and stability theory where
previously there was a broad range of bespoke techniques of proof for
specific cases. This framework has yielded new finite element spaces
and new stability proofs. In this paper, we make use of this framework
to design new numerical schemes for the rotating shallow-water
equations. The approach makes clear what kind of geometric objects are
being dealt with in the equations, and whether they should be
interpreted as point values, edge integrals, or cell integrals. In
particular, in makes it clear which terms involve the metric (and are
necessarily more complicated, especially on unstructured grids), and
which do not (and hence should be easy to discretise in a simple and
efficient way). Furthermore, the exterior derivative $\diff$ is a very
simple operation, since it requires no metric information; this should
be reflected by choosing a simple discrete form of $\diff$. The
fundamental reason why curl-grad and div-curl both vanish is because
$\diff^2=0$; fundamentally these are very simple properties and this
should be reflected in the discretisation.

The rest of this paper is structured as follows. In Section \ref{d
  forms} we provide a ``hands-on'' introduction to the calculus of
differential forms, then write the rotating shallow water equations in
differential form notation. In Section \ref{primal grid formulation}
we describe our primal grid finite element exterior calculus
formulation of the shallow water equations, and in Section \ref{dual
  grid formulation} we describe our primal/dual grid formulation. In
Section \ref{numerical}, we present some numerical results obtained
using these methods. Finally, in Section \ref{summary and outlook}, we
provide a summary and outlook.

\section{Differential forms on manifolds}
\label{d forms}
In this section, we introduce the required concepts from the language
of differential forms, in an informal manner where we shall quote a
number of basic results without proof. For more rigorous definitions,
the reader is referred to \cite{MaRa1999,ArFaWi2006,Ho2011}. We then
combine these concepts to write the rotating shallow-water equations
on the sphere in differential form notation.

\subsection{Differential form preliminaries}

\paragraph{Solution domain} We shall consider the case in which the
solution domain $\Omega$ is a closed compact oriented two-dimensional
surface. In applications the main surfaces of interest are the surface
of the sphere, or a rectangle in the $x-y$ plane with periodic
boundary conditions in both Cartesian directions. For brevity of
notation we do not consider domains with boundaries; this avoids the
need to include boundary terms when integrating by parts, although
they can easily be included.

It is useful to define local coordinates on a patch $U\subset \Omega$
\emph{via} an invertible mapping $\phi_U:U \to V\subset\mathbb{R}^2$; the
coordinates of a point $\MM{x}\in U$ are given by the value of
$(x^1,x^2)=\phi_U(\MM{x})$.

\paragraph{Vector fields} The tangent space $T_{\MM{x}}\Omega$
associated with a point $\MM{x}\in\Omega$ is the space of vectors that
are tangent to $\Omega$ at $\MM{x}$. A vector field $\MM{u}$ on
$\Omega$ is a mapping from each point $\MM{x}\in\Omega$ to the tangent
space $T_{\MM{x}}\Omega$, \emph{i.e.} it is a velocity field that is
everywhere tangent to $\Omega$. We denote $\mathfrak{X}(\Omega)$ as
the space of vector fields on $\Omega$. On a coordinate patch $U$ with
coordinates $(x^1,x^2)$, a vector field $\MM{u}$ can be expanded in
the basis $(\partial/\partial x^1,\partial/\partial x^2)$ as
$\sum_{i=1}^2u^i\pp{}{x^i}$.

\paragraph{Differential forms} In this paper we shall make use of
three types of differential forms: 0-forms (which are simply
scalar-valued functions), 1-forms, 2-forms. In general, 1-forms are
used to compute line integrals, and 2-forms are used to compute
surface integrals. We shall denote $\Lambda^k$ as the space of
$k$-forms.

\paragraph{1-forms} 
Cotangent vectors at a point $\MM{x}\in\Omega$ are the dual objects to
tangent vectors, \emph{i.e.}, linear mappings from $T_{\MM{x}}\Omega$
to $\mathbb{R}$. The space of cotangent vectors at $\MM{x}$ is written
as $T^*_{\MM{x}}\Omega$. A differential 1-form $\omega$ 
assigns a cotangent vector $v\in T^*_{\MM{x}}\Omega$ to each point
$\MM{x}\in\Omega$. This means that each 1-form $\omega$ defines a
mapping from vector fields $\MM{u}$ to scalar functions, with the
corresponding scalar function $\omega(\MM{u})$ evaluated at a point
$\MM{x}$ being written as $\omega(\MM{u})(\MM{x})$.

In local coordinates, we can obtain a basis $(\diff x^1,\diff x^2)$
for 1-forms that is dual to the basis $(\partial/\partial
x^1,\partial/\partial x^2)$ for vector fields, and we can
expand 1-forms as
\begin{equation}
\omega = \sum_i\omega_i\diff x^i.
\end{equation}
A 1-form can be integrated along a one-dimensional oriented curve
$C\subset\Omega$ using the usual definition of line integration
\begin{equation}
\int_C \omega = \int_{\phi_U(C)}\sum_i\omega_i\diff x^i.
\end{equation}
Due to the change-of-variables formula for integration, this
definition is coordinate independent.

\paragraph{$2$-forms}
A 2-form  is a function that assigns a skew-symmetric
bilinear map $T_{\MM{x}}\Omega\times T_{\MM{x}}\Omega \to \mathbb{R}$
on the tangent space $T_{\MM{x}}\Omega$ to each point
$\MM{x}\in\Omega$, that is used to define surface integration on
$\Omega$.

\paragraph{Wedge product}
The wedge product of a $k$-form and a $l$-form is a $k+l$-form,
and satisfies the following conditions:
\begin{enumerate}
\item Bilinearity:
\begin{equation}
(a\alpha + b\beta)\wedge \omega 
= a(\alpha\wedge \omega) + b(\beta\wedge \omega), \qquad
\alpha\wedge(a\omega + b\epsilon)
= a(\alpha\wedge \omega) + b(\alpha\wedge \epsilon),
\end{equation}
where $a$ and $b$ are scalars, $\alpha$ and $\beta$ are $k$-forms
and $\omega$ and $\epsilon$ are $l$-forms.
\item Anticommutativity:
\begin{equation}
\omega \wedge \gamma = (-1)^{kl}\gamma \wedge \omega,
\end{equation}
for a $k$-form $\omega$ and an $l$-form $\gamma$, and
\item Associativity:
\begin{equation}
(\omega \wedge \gamma) \wedge \kappa = \omega \wedge (\gamma \wedge \kappa).
\end{equation}
\end{enumerate}
Here, we only consider two cases:
\begin{enumerate}
\item For two 1-forms $\alpha$ and $\beta$ on
$\Omega$, the wedge product $\alpha\wedge \beta$ is a 2-form on
$\Omega$, defined by
\begin{equation}
\alpha\wedge\beta (\MM{v}_1,\MM{v}_2)
 = \alpha(\MM{v}_1)\beta(\MM{v}_2)-
\alpha(\MM{v}_2)\beta(\MM{v}_1),
\end{equation}
for all pairs of vector fields $\MM{v}_1$, $\MM{v}_2$.
\item 
The
wedge product of a scalar function (0-form) $f$ with a $k$-form
$\omega$ is simply the arithmetic product:
\begin{equation}
f\wedge \omega = f\omega.
\end{equation}
\end{enumerate}
From these properties it may be deduced that the wedge product of two
arbitrary 1-forms $\omega$, $\gamma$ may be written in coordinates
as 
\begin{equation}
\omega \wedge \gamma = \alpha \diff x^1\wedge \diff x^2,
\end{equation}
for some scalar function $\alpha$, and hence this is the general form
for 2-forms in Cartesian coordinates.

\paragraph{Integration of 2-forms and the surface form}

In coordinates, the integral of a 2-form $\omega=\alpha \diff x^1
\wedge \diff x^2$ over a 2-dimensional submanifold $M\subset\Omega$
\begin{equation}
\int_M \omega = \int_{\phi_U(M)}\alpha \diff x^1\diff x^2.
\end{equation}
This definition is coordinate independent, due to the change of
variables formula. For chosen oriented coordinates, there exists a
unique $\alpha_S$ such that this integral provides the surface area of
each submanifold $M$ (using a suitable Riemannian metric on $M$, for
example using the Euclidean metric inherited from the three
dimensional space in which $\Omega$ is embedded). The corresponding
2-form is called the surface form, and is written
\begin{equation}
\diff S = \alpha_S\diff x^1\wedge \diff x^2.
\end{equation}
This definition is also coordinate independent, and hence we may
write any 2-form in the form $\omega=\beta \diff S$, for a scalar
function $\beta$.

\paragraph{Contraction with vector fields}
Contractions of $k$-forms with vector fields are used to calculate
advective fluxes. In general, the contraction of a vector field $\MM{u}$
with a $k$-form $\omega$ results in a (k-1)-form, denoted
$\contract{\MM{u}}\omega$. The contraction of a vector field $\MM{u}$
with a $0$-form is zero, and with a 1-form $\omega$ is simply the
scalar function
\begin{equation}
\contract{\MM{u}}\omega(\MM{x}) = \omega(\MM{u})(\MM{x}).
\end{equation}
In general, the contraction is linear, \emph{i.e.}
$\contract{\MM{u}}(a(x)\omega_1+b(x)\omega_2)
=a(x)\contract{\MM{u}}\omega_1+b(x)\contract{\MM{u}}\omega_2$, for two
scalar functions $a(x)$ and $b(x)$, and two $k$-forms $\omega_1$ and
$\omega_2$. The contraction of a vector field $\MM{v}$ with a 2-form
$\omega$ is the 1-form $\contract{\MM{u}}\omega$ defined by
\begin{equation}
\left(\contract{\MM{u}}\omega\right)(\MM{v}) = \omega(\MM{u},\MM{v}),
\end{equation}
for all vector fields $\MM{v}$, and so may be written in coordinates
as
\begin{equation}
\contract{\MM{u}} \alpha \diff x^1\wedge \diff x^2
= \alpha\left(u_1 \diff x^2 - u_2 \diff x^1\right).
\end{equation}

\paragraph{Identification of vector fields with 1-forms}
To write equations of motion using differential forms it is necessary
to make an identification between vector fields and differential forms
(vector field proxies).  In this framework we shall make use of two
different identifications of vector fields with 1-forms\footnote{In
  general, on an $n$-dimensional manifold $M$, there is one
  identification of vector fields with 1-forms, and one with $n-1$
  forms, but we are working with 2-dimensional manifolds here.}.
\begin{enumerate}
\item $\MM{u}\in\mathfrak{X}(\Omega)\mapsto \tilde{u} \in \Lambda^1$ defined by
\begin{equation}
\tilde{u}(\MM{v})(\MM{x}) = \langle\MM{u},\MM{v}\rangle(\MM{x}), \quad
\forall \MM{v}\in \mathfrak{X}(\Omega),
\end{equation}
where $\langle \MM{u},\MM{v}\rangle$ is an inner product on
$\mathfrak{X}(\Omega)$. In coordinates, with
$\MM{u}=\sum_iu^i\pp{}{x^i}$, $\MM{v}=\sum_iv^j\pp{}{x^i}$, we have
\begin{equation}
\langle \MM{u},\MM{v}\rangle = \sum_{ij}u^ig_{ij}v^j,
\end{equation}
where $g_{ij}$ is the metric tensor associated with the inner product,
and hence
\begin{equation}
\tilde{u} = \sum_i\tilde{u}_i\diff x^i, \quad \mbox{ with }
\tilde{u}_i = \sum_jg_{ij}u^j.
\end{equation}

This identification is used to compute circulation integrals
\begin{equation}
\int_C \MM{u}\cdot \diff\MM{x} = \int_C \tilde{u},
\end{equation}
along curves $C\subset \Omega$, and hence is associated with the curl
operator. We shall use the notation $\tilde{u}$ to denote the 1-form
obtained from a vector field $\MM{u}$ using this identification.
\item The second identification is written using the
contraction with $\diff{S}$,
\begin{equation}
\MM{u} \mapsto \contract{\MM{u}}\diff{S},
\end{equation}
and is used to compute flux integrals
\begin{equation}
\int_C \contract{\MM{u}}\diff{S}
\end{equation}
across curves $C\subset \Omega$, and hence is associated with the
divergence operator.
\end{enumerate}

\paragraph{Exterior derivative}
The differential operator (exterior derivative) $\diff$ neatly encodes
all of the vector calculus differential operators \emph{e.g.}, div,
grad, curl \emph{etc.} In general, $\diff$ maps $k$-forms to
$(k$+$1)$-forms, and satisfies:
\begin{enumerate}
\item For scalar functions $f$, $\diff f=\sum_i\pp{f}{x^i}\diff x^i$
  in coordinates.
\item Product rule: for $\omega\in\Lambda^k$ and $\gamma\in
  \Lambda^l$, $\diff(\omega \wedge \gamma)=(\diff\omega)\wedge\gamma +
  (-1)^k\omega\wedge(\diff\gamma)$.
\item Closure: $\diff (\diff\omega) = \diff^2\omega = 0$.
\end{enumerate}

Standard vector calculus differential operators on scalar functions
$f$ and vector fields $\MM{u}$ defined on $\Omega$ are obtained from
the two vector field proxies:
\begin{eqnarray}
\diff f & = & \widetilde{\nabla f}, \\
\diff f & = & -\contract{(\nabla^\perp f)}\diff{S}, \\
\diff(\widetilde{u}) & = & 
(\nabla^\perp\cdot\MM{u})\diff{S}, \\
\diff(\contract{\MM{u}}\diff S) & = & \nabla\cdot\MM{u}\diff{S},
\end{eqnarray}
where $\nabla$, $\nabla^\perp=\hat{\MM{k}}\times \nabla$,
$\nabla^\perp\cdot=\hat{\MM{k}}\cdot\nabla\times$ and $\nabla\cdot$ are
vector calculus differential operators defined intrinsically on the
two-dimensional surface $\Omega$ with $\hat{\MM{k}}$ being the unit
vector normal to the manifold $\Omega$. The closure property
$\diff^2=0$ then leads to the following vector identities for the two
identifications of vector fields with 1-forms:
\begin{eqnarray}
\label{complex 1}
0&=&\diff^2 f = \diff(\widetilde{\nabla f})
= (\nabla^\perp\cdot\nabla f)\diff{S}, \\
\label{complex 2}
0&=& \diff^2 f = -\diff(\contract{\nabla^\perp f}\diff{S}) =
-(\nabla\cdot\nabla^\perp f) \diff{S}.
\end{eqnarray}
These identities are crucial for geophysical applications since they
dictate the scale separation between slow divergence-free and fast
divergent dynamics.

\paragraph{Stokes' theorem and integration by parts}

The general form of Stokes' theorem for $\omega\in \Lambda^k$ is 
\begin{equation}
\int_M \diff\omega = \int_{\partial M}\omega,
\end{equation}
where $M$ is a $k+1$-dimensional submanifold of $\Omega$, $\omega$ is
a $k$-form (and hence $\diff\omega$ is a $(k+1)$-form), and $\partial
M$ is the $k$-dimensional submanifold corresponding to the boundary of
$M$. Combining Stokes' theorem with the product rule provides the
integration by parts formula
\begin{equation}
\int_M (\diff\omega)\wedge\gamma = \int_{\partial M}\omega\wedge\gamma
+(-1)^{k-1}\int_M\omega\wedge(\diff\gamma),
\end{equation}
for $\gamma \in \Lambda^l$, for ($k$+$l$+1)-dimensional manifolds $M$ with
($k$+$l$)-dimensional boundary $\partial M$.

\paragraph{Hodge star}
The Hodge star operator $\star$ maps from $k$-forms to (2-$k$)-forms,
and is defined relative to a chosen metric on the manifold $\Omega$
(in our case we use the usual Euclidean metric from $\mathbb{R}^3$).
 It is used in this paper to write the $L_2$-inner product between
two $k$-forms $\omega$ and $\gamma$ by
\begin{equation}
\langle \omega, \gamma \rangle_{L_2} = \int \omega \wedge \star\gamma
= \int \gamma \wedge \star\omega,
\end{equation}
and is also used to write the Coriolis term. The Hodge star is linear
(\emph{i.e.}, $\star(a(x)\omega+b(x)\gamma)=a(x)\star\omega +
b(x)\star\gamma$ for scalar functions $a$, $b$ and $k$-forms $\omega$,
$\gamma$). 

Here we omit the intrinsic definition and just state 
the effect of the Hodge star on vector field proxies:
\begin{eqnarray}
\star f & = & f\diff{S}, \\
\star f\diff{S} & = & f, \\
\star\tilde{u} & = & \contract{\MM{u}}\diff{S} = 
\widetilde{u^{\perp}}, \\
\star\contract{\MM{u}}\diff{S} & = & -\tilde{u}
= \contract{\MM{u}^\perp}\diff{S},
\end{eqnarray}
where $\MM{u}^\perp=\hat{\MM{k}}\times\MM{u}$, which is also a vector
field on $\Omega$.  For two vector fields $\MM{w}$ and $\MM{u}$, we
have
\begin{equation}
\langle\MM{w},\MM{u}\rangle\diff S=\tilde{w}\wedge\star\tilde{u}.
\end{equation}
From the presence of $\hat{\MM{k}}\times$ in these formulas it becomes
clear that the Hodge star is useful for expressing the Coriolis
term. Note that $\star\star=\Id$ for 0- and 2-forms, and
$\star\star=-\Id$ for 1-forms. Since $\contract{\MM{u}}\diff{S}$ is
quite a lengthy notation, we shall use $\star\tilde u$ to denote the
second vector field proxy for a vector field $\MM{u}$.

\paragraph{Dual differential operator}
We define $\delta$ as the dual differential operator from
${\Lambda}^k$ to ${\Lambda}^{k-1}$ that is dual to
$d$, \emph{i.e.}
\begin{equation}
\int_{\Omega} \gamma \wedge \star \delta \omega = \int_{\Omega}
\diff\gamma \wedge \star\omega, \quad \forall
\gamma\in{\Lambda}^{k-1}, \, \omega\in{\Lambda}^k.
\end{equation}
We note that $\delta^2\omega=0$, since
\begin{eqnarray}
\int \gamma \wedge \star \delta^2\omega & = &
\int \diff\gamma \wedge \star \delta \omega \\
& = & \int \diff^2\gamma \wedge \star \omega \\
& = & 0 \qquad \forall \gamma \in \Lambda^{k-2}, \,
\omega \in \Lambda^k.
\end{eqnarray}

\subsection{Rotating shallow-water equations in differential form
  notation}
We have now established enough notation to write the rotating
shallow-water equations on $\Omega$ in differential form notation,
which will be our starting point to develop finite element
approximations in Section \ref{FEEC}. We begin from the following form
of the rotating shallow-water equations:
\begin{eqnarray}
  \pp{}{t}\MM{u} + (\zeta+f)\MM{u}^\perp + \nabla
\left(g(D+b) + \frac{1}{2}|\MM{u}|^2\right)
  & = & 0, \\
  \pp{}{t}D + \nabla\cdot(\MM{u}D) & = & 0, 
\end{eqnarray}
where $\MM{u}$ is the velocity, $\zeta=\hat{\MM{k}}\cdot\nabla\times
\MM{u}=\nabla^\perp\cdot\MM{u}$ is the vorticity, $D$ is the layer
depth, $b$ is the height of the bottom surface, $g$ is the
acceleration due to gravity, $f$ is the Coriolis parameter and
$\MM{u}^\perp=\hat{\MM{k}}\times\MM{u}$. This form of the equations,
is known in the numerical weather prediction community as the ``vector
invariant form'' \cite{ArLa1981}. It is widely used because it is easy
to relate to the vorticity budget; we shall show that it leads in a
straightforward computation to local conservation of potential
vorticity $q=(\zeta+f)/D$, and that this computation only involves
properties of $\diff$, $\wedge$ and $\star$ that can be preserved by
the finite element exterior calculus. It is also easy to relate to the
energy budget, and the demonstration of conservation of energy also
only involves these properties so this can again be preserved by the
finite element exterior calculus.

Using the notation that we have described above, we can rewrite these
equations as
\begin{eqnarray}
\label{u eqn}
\pp{}{t}\tilde{u} + \star
\underbrace{\tilde{u}(\zeta+f)}_{\tilde{Q}} + \diff\left(g(D+b)+
  K\right) & = & 0, \\
\label{D eqn}
\pp{}{t}D\diff{S} + \diff\star(\tilde{u}D) & = & 0, \\
\label{zeta eqn}
\diff\tilde{u} & = & \zeta\diff{S},
\end{eqnarray}
where $K=|\MM{u}|^2/2=\tilde{u}\wedge\star\tilde{u}\diff S/2$.  The choice of a
1-form for equation \eqref{u eqn} is natural since we can integrate it
to obtain a circulation equation around a closed loop $C$,
\begin{equation}
\dd{}{t}\int_C\tilde{u} + \int_C \star{\tilde{Q}} +
\underbrace{\int_C \diff\left(g(D+b)-K\right)}_{=0}=0,
\end{equation}
or apply $\diff$ to obtain an evolution equation for the vorticity.
The choice of a 2-form for equation \eqref{D eqn} is natural since we
can integrate it to obtain a mass budget in a control area $A$,
\begin{equation}
\dd{}{t}\int_A D\diff{S} + \int_{\partial A} \star\tilde{u}D=0,
\end{equation}
where $\partial A$ is the boundary of $A$. Note that equation
\eqref{u eqn} naturally makes use of the circulation 1-form vector
field proxy $\tilde{u}$ whilst equation \eqref{D eqn} make use of the
other 1-form vector field proxy $\star\tilde{u}$. When we choose
finite element spaces in the next section, we will need to choose one
proxy or the other since they come with different interelement
continuity requirements for $\MM{u}$.

Applying $\diff$ to equation \eqref{u eqn} and making use of
$\diff^2=0$ and the definition of Hodge star immediately gives the
vorticity equation
\begin{equation}
\pp{}{t}\zeta\diff{S} + \diff\star(\tilde{u}(\zeta+f)) =  0,
\end{equation}
which is in the same flux form as the mass equation (equation \eqref{D
  eqn}).  The potential vorticity (PV) $q$ is defined from
\begin{equation}
(\zeta + f)\diff{S} = q D\diff{S},
\end{equation}
and hence we obtain the law of conservation of potential vorticity
\begin{eqnarray}
\label{pv}
\pp{}{t} \left(qD\diff{S}\right) + \diff\star(\tilde{u}qD)
=0.
\end{eqnarray}
Note that if $q$ is constant then
\begin{equation}
\left(\pp{q}{t}D\diff{S}\right) + q\left(
\underbrace{
\pp{}{t}D\diff{S}+
\diff\star(\tilde{u}D)}
_{=0}\right)=0 \implies \pp{q}{t}=0,
\end{equation}
which means that $q$ remains constant. This is what is meant by 
consistency of equation \eqref{pv} with equation \eqref{D eqn}.

Our goal is to design a framework for finite element discretisations
that has $\MM{u}$ and $D$ as the prognostic variables, yet preserves
the conservation law structure of equations \eqref{D eqn} and
\eqref{pv}.  Furthermore, we shall show how stabilisations for these
conservation laws (which are required for meteorological applications)
can be incorporated into this framework.

\section{Finite element exterior calculus formulation}
\label{FEEC}
In this section we develop finite element exterior calculus
approximations to equations (\ref{u eqn}-\ref{D eqn}) and demonstrate
their conservation properties.

\subsection{Finite element spaces}
The fundamental idea of finite element exterior calculus applied to
the rotating shallow-water equations is to choose finite element
spaces for the discretised variables $\MM{u}^h$, $\zeta^h$ and $D^h$
such that the operator $\diff$ maps from one space to another, so that
the vector calculus identities \eqref{complex 1} and \eqref{complex 2}
still hold. The difficulty is that continuity of $\MM{u}^h$ in the
normal direction across element boundaries is required to compute
$\diff\star\tilde{u}^h$ (required to compute fluxes), whilst
continuity in the tangent direction across element boundaries is
required to compute $\diff\tilde{u}^h$ (required to compute the
relative vorticity $\zeta^h$).  On a single grid, we cannot have both,
and thus we must choose to construct finite element spaces such that
only one of \eqref{complex 1} or \eqref{complex 2} hold in the strong
form, and the other will hold in the weak form after integrating by
parts. This amounts to choosing one of $D^h$ and $\zeta^h$ to have a
continuous finite element space and the other to have a discontinuous
space. A discontinuous space allows for discontinuous Galerkin methods
which are locally conservative and allow for shape preserving
advection schemes, and in meteorogical applications it is more
important that these schemes are available for $D^h$ than $\zeta^h$,
so we choose to hold \eqref{complex 1} in the strong form. Later, in
setting up the primal-dual grid formulation, we shall introduce a dual
grid for which \eqref{complex 2} holds in the strong form,
consistently with the weak form on the primal grid.

\paragraph{Finite element differential form spaces} 
Having made the choice to hold \eqref{complex 1} in strong form, we
need to choose finite element spaces for $\zeta^h$, $\MM{u}^h$, and $D^h$,
denoted $V^0$, $V^1$ and $V^2$ respectively.  This choice defines equivalent
subspaces $\hat{\Lambda}^k\subset\Lambda^k$, $k=1,2,3$,
given by
\begin{eqnarray}
\hat{\Lambda}^0 &=& V^0, \\
\hat{\Lambda}^1 &=& \{ \star\tilde{u}^h: 
\MM{u}^h\in V^1\}, \\
\hat{\Lambda}^2 &=& \{ D^h\diff S: D^h \in V^2\},
\end{eqnarray}
We require that $\diff$ maps from $\hat{\Lambda}^1$ into
$\hat{\Lambda}^2$, and that $\diff$ maps from $\hat{\Lambda}^0$ into
$\hat{\Lambda}^1$ (in particular, onto the kernel of $\diff$ in
$\hat{\Lambda}^1$), which implies that $V^0$ is a continuous finite
element space, $V^1$ is div-conforming (\emph{i.e.} $\MM{u}^h\in V^1$
has continuous normal components across element boundaries), and $V^2$
is a discontinuous finite element space. This is expressed in
the following diagram,
\begin{equation}
 \begin{CD}
   \hat{\Lambda}^0 @> \diff >> \hat{\Lambda}^1 @>\diff >> \hat{\Lambda}^2 .
 \end{CD}
\end{equation}
Numerous examples of $(V^0,V^1,V^2)$ satisfying these properties
exist, for example $V^0=P(k+1)$ (degree $k$+$1$ polynomials in each
triangular element with $C^0$ continuity between elements),
$V^1=BDM(k)$ (degree $k$ vector polynomials with continuous normal
components across element edges, known as the $k$th
Brezzi-Douglas-Marini space), $V^2=P(k-1)_{DG}$ (degree $k-1$
polynomials with no inter-element continuity requirements). For more
details of the families of finite element spaces that satisfy these
conditions, see \cite{ArFaWi2006}.

In practice, to implement these schemes on a computer it is necessary
to expand functions in the finite element spaces in a basis, to obtain
discrete vector systems, but most techniques of proof avoid choosing a
particular basis since this usually obscures what is happening.

\paragraph{Discrete dual differential operator}
Whilst $\diff$ is identical to the operator used in the unapproximated
equations, we must approximate the dual operator $\delta$.  We define
$\delta^h$ as the discrete dual differential operator from
$\hat{\Lambda}^k$ to $\hat{\Lambda}^{k-1}$ that is dual to $\diff$,
\emph{i.e.}
\begin{equation}
\int_{\Omega} \gamma^h \wedge \star \delta^h \omega^h = 
\int_{\Omega}
\diff\gamma^h \wedge \star\omega^h, \quad \forall
\gamma^h\in\hat{\Lambda}^{k-1},\,\omega^h\in \hat{\Lambda}^k.
\end{equation}
Note that $\delta^h$ from $\hat{\Lambda}^k$ to $\hat{\Lambda}^{k-1}$ is
only an approximation to the dual differential operator defined
from $\Lambda^k$ to $\Lambda^{k-1}$, but that it still satisfies
$(\delta^h)^2=0$.

\paragraph{Discrete Helmholtz decomposition} As discussed in 
\cite{ArFaWi2006}, if $\diff$ maps from $\hat{\Lambda}^0$ onto 
the kernel of $\diff$ in $\hat{\Lambda}^1$, then there is a
discrete Helmholtz decomposition and any 1-form 
$\omega^h\in \hat{\Lambda}^1$ can be written
as
\begin{equation}
\omega^h = \diff \psi^h + \delta^h \phi^h
+ \mathfrak{h}^h,
\end{equation}
where $\psi^h\in\hat{\Lambda}^0$, $\phi^h\in\hat{\Lambda}^2$, and
$\mathfrak{h}^h \in \mathcal{H}$, where $\mathcal{H}\subset
\hat{\Lambda}^1$ is the space of discrete harmonic 1-forms
given by
\begin{equation}
\mathcal{H} = \{\mathfrak{h}^h\in \hat{\Lambda}^1:
\diff\mathfrak{h}^h=0, \, \delta^h\mathfrak{h}^h=0\},
\end{equation}
which has the same dimension as the space of continuous harmonic
1-forms on $\Omega$ (and which has dimension 0 for the surface of a
sphere).

\paragraph{Construction of global finite element spaces by pullback}

We construct the spaces $V^k$, $k=0,1,2$, and hence $\hat{\Lambda}^k,
k=0,1,2$, by dividing $\Omega$ (or a piecewise polynomial
approximation of $\Omega$) into elements, restricting $V^k$ to some
choice of polynomials on each element, and specifying the interelement
continuity (from the requirements of $\diff$ discussed above).  This
is most easily done by defining a reference element $\hat{e}$ where
integrals are computed, and a choice of polynomial function spaces
$\hat{\Lambda}^k(\hat{e})$, $k=0,1,2$ such that 
\begin{equation}
 \begin{CD}
   \hat{\Lambda}^0(\hat{e}) @> \diff >> \hat{\Lambda}^1(\hat{e}) @>\diff
   >> \hat{\Lambda}^2(\hat{e}).
 \end{CD}
\end{equation}
For each element $e$, we then define the element mapping
$\eta_e:\hat{\Lambda}^k(\hat{e})\to \hat{\Lambda}^k(e)$, where
$\hat{\Lambda}^k(e)$ is the space $\hat{\Lambda}^k$ restricted to
$e$. The mapping $\eta_e$ is a diffeomorphism, usually expanded in
polynomials. This then defines a mapping from $\Lambda^k(e)$ 
to $\Lambda^k(\hat{e})$ \emph{via} pullback
\begin{equation}
\omega^h \mapsto \eta^*_e\omega^h,
\end{equation}
where the pullback $\eta^*\omega$ of a $k$-form $\omega$ by a
diffeomorphism $\eta$ is defined by
\begin{equation}
\int_{M} \eta^*\omega = \int_{\eta(M)}\omega,
\end{equation}
for all integrable $k$-dimensional submanifolds $M$. The pullback
operator satisfies two useful properties:
\begin{enumerate}
\item
Pullback $\eta^*$ commutes with $\diff$:
$\diff \eta^*\omega=\eta^*\diff\omega$.
\item Pullback is compatible with the wedge
product:
$\eta^*(\alpha\wedge\omega)=(\eta^*\alpha)\wedge(\eta^*\omega)$.
\end{enumerate}

We define $\hat{\Lambda}^k(e)$ by
\begin{equation}
\hat{\Lambda}^k(e) = \left\{\omega^h\in \Lambda^k(e):\eta^*_e\omega^h \in \hat{\Lambda}^k(\hat{e})\right\}.
\end{equation}
In coordinates, the pullback $\eta^*_e\gamma^h\in
\hat{\Lambda}^0(\hat{e})$ of $\gamma^h \in \hat{\Lambda}^0(e)$ is
$\gamma^h\circ \eta_e^{-1}$. The pullback
$\eta^*_e\star\tilde{u}^h\in\hat{\Lambda}^1(\hat{e})$ of
$\star\tilde{u}^h\in \hat{\Lambda}^1(e)$ defines the (contravariant)
Piola transformation \cite{BrFo1991}
\begin{equation}
\eta_e^*\star\tilde{u}^h = \tilde{\hat{{u}}}^h \implies
u^h_i\circ \eta^{-1}_e =
\frac{1}{\det J_e}
\sum_j\left(J_e\right)_{ij}\hat{{u}}_j^h, \qquad J_e = \pp{\eta_e}{\hat{\MM{x}}},
\end{equation}
where $u^h_i$, $\hat{u}^h_i$ are the components of $\MM{u}^h$,
$\hat{\MM{u}}^h$ in our chosen coordinate systems, and $(J_e)_{ij}$
are the components of $J_e$. The pullback of
$D^h\diff{S}\in\hat{\Lambda}^2(e)$ defines the scaling transformation
\begin{equation}
\eta_e^*(D^h\diff{S}) = \hat{D}^h\diff\hat{S}
\implies D^h\circ \eta^{-1}_e = \hat{D}^h\left(\det J_e\right).
\end{equation}
Since pullback commutes with $\diff$, it is sufficient to check that
$\diff$ maps from $\hat{\Lambda}^k(\hat{e})$ to
$\hat{\Lambda}^{k+1}(\hat{e})$, $k=0,1$ to guarantee that it maps from
$\hat{\Lambda}^k$ to $\hat{\Lambda}^{k+1}$ (provided that the spaces
have sufficient interelement continuity that $\diff$ is defined). A
technicality for the $\Lambda^2$ case is that if $\eta_e$ is affine
(\emph{i.e.} the mesh triangles are flat), $\det J_e$ is constant, and
we obtain the same finite element space if we transform $D^h$ as a
0-form, \emph{i.e.} $D^h\circ \eta_e=\hat{D}^h$.  This simplifies some of
the expressions as will be discussed in the next section. In the
non-affine case, we may also transform $D^h$ as a 0-form, but this
requires further modifications to the framework \cite{BoRi2008}.

 For a description of an implementation of the Piola transformation
 and global assembly of $V^1$, see \cite{RoKiLo2009}, and for a
 description of an implementation on manifold meshes see
 \cite{Ro+2013}.

\subsection{Finite element discretisation: primal grid formulation}
\label{primal grid formulation}
In this section we provide a finite element (semi-discrete continuous
time) discretisation of the shallow water equations, and show that it
conserves mass, energy, potential enstrophy and potential vorticity. We then
show how to introduce dissipative stabilisations such that mass and
potential vorticity are still conserved.

Since $\hat{\Lambda}^2$ is a discontinuous space, we cannot apply
$\diff$ to $D^h$ and must instead adopt the weak form.  This is done by
taking the wedge product of Equation \eqref{u eqn} with a test 1-form
$\star\tilde{w}^h\in \hat{\Lambda}^1$, integrating over the domain
$\Omega$, integrating by parts (with vanishing boundary term since
there are no boundaries), and multiplying by -1:
 \begin{equation}
\dd{}{t}
\int_{\Omega}(\star\tilde{w}^h)\wedge\star(\star\tilde{u}^h)
+\int_{\Omega}(\star\tilde{w}^h)\wedge(\star\tilde{Q}^h)
 + \int_{\Omega}\diff(\star\tilde{w}^h)\wedge
\left(g(D^h+b^h)+K^h\right) = 0, \qquad \forall \star\tilde{w}^h 
\in \hat{\Lambda}^1.
\label{u cg}
\end{equation}
To obtain the Galerkin finite element approximation of this equation
we restrict $\star\tilde{u}^h$, $\star\tilde{w}^h$ to the finite
element space $\hat{\Lambda}^1$, and $D^h\diff{S}$ and $b\diff{S}$ to
the finite element space $\hat{\Lambda}^2$. Similarly, we write the
weak form of equation \eqref{D eqn} as
\begin{equation}
\label{D cg}
\dd{}{t}\int_{\Omega} \phi^h \wedge D^h\diff S +
\int_{\Omega} \phi^h \wedge  \diff (\star\tilde{F}^h)
= 0, \quad \forall \phi^h\diff S \in
\hat{\Lambda}^2,
\end{equation}
where $\star\tilde{F}^h\in \hat{\Lambda}^1$ is the mass flux, and the
Galerkin finite element approximation is obtained by restricting
$D^h\diff{S}$ and $\phi^h\diff{S}$ to the finite element space
$\hat{\Lambda}^2$.  To close the system, it remains to define the
vorticity flux $\star\tilde{Q}^h$ and the mass flux
$\star\tilde{F}^h$. Before we do that, we note the following property of
the discrete equations (\ref{u cg}-\ref{D cg}).
\begin{remark}[Topological terms]
  It is useful to note that apart from the $\diff/\diff t$ terms, all
  of the terms in Equations (\ref{u cg}-\ref{D cg}) are purely
  topological. To see this, taking the second term in Equation
  \eqref{u cg} as an example, we write the integral as a sum over
  elements,
  \begin{eqnarray}
    \int_{\Omega} (\star\tilde{w}^h)\wedge(\star\tilde{Q}^h)
    &=& \sum_e \int_e (\star{\tilde{w}^h})\wedge(\star{\tilde{Q}^h}), \\
    &=& \sum_e \int_{\hat{e}}g^*_e
\left((\star{\tilde{w}^h})\wedge(\star{\tilde{Q}^h})\right), \\
    &=& \sum_e \int_{\hat{e}}
g^*_e(\star{\tilde{w}^h})\wedge g^*_e(\star{\tilde{Q}^h}), \\
 &=& \sum_e \int_{\hat{e}} (\star\hat{\tilde{w}}^h)\wedge(\star\hat{\tilde{Q}}^h).
  \end{eqnarray}
  Since we use $\hat{\tilde{w}}^h$ and $\hat{\tilde{Q}}^h$ as our
  computational variables, this expression has no factors of $J_e$,
  and hence the integral over each element is independent of the
  element coordinates: the global integral only depends on the mesh
  topology. Similarly, for an integral of the form (which corresponds
  to the form of the pressure gradient, as well as the mass flux term
  upon exchange of the trial and test functions)
  \begin{eqnarray}
    \int_{\Omega} \diff(\star\tilde{w}^h)\wedge \phi^h & = &
    \sum_e \int_{\hat{e}}g^*_e\left(
\diff(\star\tilde{w}^h)\wedge \phi^h\right),  \\
&=&     \sum_e \int_{\hat{e}}g^*_e(\diff\star\tilde{w}^h)\wedge g^*_e\phi^h,\\
&=&     \sum_e \int_{\hat{e}}\diff g^*_e(\star\tilde{w}^h)\wedge g^*_e\phi^h,\\
&=&     \sum_e \int_{\hat{e}}\diff(\star\hat{\tilde{w}}^h)\wedge \hat{\phi^h},
  \end{eqnarray}
  where we have made use of $\diff$ commuting with pullback in the
  last line to obtain an expression that is independent of $J_e$.

  The $\diff/\diff t$ terms are not purely topological since they
  involve an extra $\star$ in Equation \eqref{u cg} and a factor of 
  $\diff S$ in Equation \eqref{D cg}, which means that metric terms 
  are present.
\end{remark}
These purely topological terms lead to efficiencies since they do not
require inversion of $J_e$ (the main contribution to flops in
\emph{e.g.} the assembly of the standard weak Laplacian using
continuous finite elements)\footnote{Note that inversion of $J_e$ is
  not needed for the time derivative terms either, so the entire
  formulation can be implemented without $J_e$ inversions.},
furthermore the contributions to the integral from element $e$ can be
calculated without even needing to load in the coordinate field (which
is an important consideration when the cost of transferring data to
processors dominates the cost of performing flops).

These topological relations lead to a number of properties of the
equations, which we shall now discuss, starting with conservation of
mass. \begin{theorem}[Mass conservation]
Let $D^h\diff S$ satisfy equation \eqref{D cg}. Then $D^h\diff S$ is 
locally conserved. 
\end{theorem}
\begin{proof}
Let $\phi^h$ be the indicator function for element $e$,
\emph{i.e.},
\begin{equation}
\phi^h(\MM{x}) = \left\{
\begin{array}{r l}
1 & \mbox{if }\MM{x}\in e, \\
0 & \mbox{otherwise},
\end{array}
\right.
\end{equation}
then equation \eqref{D cg} becomes
\begin{equation}
\underbrace{\dd{}{t}\int_e D^h\diff{S}}_{\textrm{change in mass in }e}
= -\int_e \diff(\star\tilde{F}^h) = 
 \underbrace{-\int_{\partial e} \tilde{F}^h.}_{\textrm{mass flux through }\partial e}
\end{equation}
Since $\star\tilde{F}^h\in \hat{\Lambda}^1$, the integral
of $\star\tilde{F}^h$ takes the same value on either side
of each of the element edges forming $\partial e$ (except with
alternate sign) and hence the flux of $D^h\diff{S}$ out of element $e$
is the same as the flux into the neighbouring elements, and 
$D^h$ is locally conserved.
\end{proof}

The vorticity is obtained from equation \eqref{zeta eqn}. If
$\star\tilde{u}^h\in \hat{\Lambda}^1$ then $\diff \tilde{u}^h$ is not
defined, so we obtain an approximation to $\zeta^h$ in
$\hat{\Lambda}^0$ by introducing a test function $\gamma^h
\in\hat{\Lambda}^0$ and integrating by parts (neglecting the surface
term as $\Omega$ is closed),
\begin{equation}
\int_{\Omega} \gamma^h \wedge \star \zeta^h =
-\int_{\Omega} \diff \gamma^h \wedge \tilde{u}^h, \qquad
\forall \gamma^h \in \hat{\Lambda}^0.
\label{zeta cg}
\end{equation}
\begin{theorem}[Discrete vorticity conservation]
Let $\star\tilde{u}^h$ satisfy equation \eqref{u cg}. Then $\zeta^h
\in \hat{\Lambda}^0$ obtained from equation \eqref{zeta cg} satisfies
\begin{equation}
\dd{}{t}\int_{\Omega}\gamma^h\wedge\zeta^h\diff S
-\int_{\Omega}\diff\gamma^h\wedge\star\tilde{Q}^h=0, 
 \qquad \forall \gamma^h \in \hat{\Lambda}^0,
\end{equation}
which is the continuous finite element approximation to the vorticity
equation in flux form. Furthermore, $\zeta^h$ is globally conserved.
\end{theorem}
\begin{proof}
Since $-\diff\gamma^h\in \hat{\Lambda}^1$ for arbitrary $\gamma^h\in
\hat{\Lambda}^0$, we may select $\star\tilde{w}^h=-\diff\gamma^h$
in equation \eqref{u cg} to obtain 
\begin{eqnarray}
\dd{}{t} \int_{\Omega}\gamma^h\wedge \zeta^h \diff S & =& 
\dd{}{t}
\int_{\Omega}\gamma^h\wedge\star \zeta^h \\
& = & 
\dd{}{t}
\int_{\Omega}-\diff\gamma^h\wedge\tilde{u}^h 
\\
& = & 
\int_{\Omega}\diff\gamma^h\wedge\star\tilde{Q}^h
 +\int_{\Omega}\underbrace{\diff^2\gamma^h}_{=0}\wedge\left(g(D^h-b)+K^h\right)\diff{S} , \\
& = & 
\int_{\Omega}\diff\gamma^h\wedge\star\tilde{Q}^h,
 \qquad \forall \gamma^h \in \hat{\Lambda}^0.
\end{eqnarray}
This is the standard continuous finite element discretisation of the
vorticity transport equation. Global conservation of vorticity is a
direct consequence of this, upon choosing $\gamma^h=1$:
\begin{eqnarray}
\dd{}{t}
\int_{\Omega}\zeta^h\diff{S} =
-\int_{\Omega}\underbrace{\diff(1)}_{=0}\wedge\star\tilde{Q}^h=0.
\end{eqnarray}
\end{proof}
Having defined a vorticity, we can define a potential vorticity
$q^h\in \hat{\Lambda}^0$ from
 \begin{equation}
\int_{\Omega} \gamma^h \wedge q^hD^h\diff{S} = \int_{\Omega}\gamma^h \wedge
(\zeta^h+f)\diff{S}.
\label{q cg}
\end{equation}
Then, similar and straightforward calculations lead to the following.
\begin{theorem}[Potential vorticity conservation]
Let $\star\tilde{u}^h$ satisfy equation \eqref{u cg} and
$D^h\diff S$ satisfy equation \eqref{D cg}. Then $q^h \in
\hat{\Lambda}^0$ obtained from equation \eqref{q cg} satisfies
\begin{equation}
\label{pv cg}
\dd{}{t}\int_{\Omega} \gamma^h \wedge q^hD^h\diff{S} 
+\int_{\Omega} \diff \gamma^h \wedge \star\tilde{Q}^h=0.
\end{equation}
This is the standard continuous finite element approximation to the
potential vorticity equation in conservation form.  Furthermore, $q$
is globally conserved \emph{i.e.}
\begin{equation}
\dd{}{t}\int_{\Omega} q^hD^h\diff S = 0.
\end{equation}
\end{theorem}
Now we return to the question of how to choose the mass and vorticity
fluxes. The following choices 
lead to energy and potential enstrophy conservation.
\begin{eqnarray}
\nonumber
\star\tilde{F}^h \in \hat{\Lambda}^1 \textrm{ with }
\int_{\Omega} (\star\tilde{w}^h) \wedge \star
(\star\tilde{F}^h) & = & \int_{\Omega} (\star\tilde{w}^h)
\wedge \star \left(D^h(\star\tilde{u}^h)\right), \\
& & \qquad \qquad \label{F eqn}
\quad \forall \star\tilde{w}^h \in \hat{\Lambda}^1, \\
\label{Q eqn}
\star\tilde{Q}^h & = & q^h(\star\tilde{F}^h).
\end{eqnarray}
Note that obtaining the mass flux from equation \eqref{F eqn} involves
solving a global, but well-conditioned matrix-vector system for the
basis coefficients of $\tilde{F}^h$.  These choices have been informed by
energy-enstrophy conserving C-grid finite difference methods designed
on latitude-longitude grids in \cite{ArLa1981} that were extended to
either energy or enstrophy conserving C-grid schemes on arbitrary unstructured
C-grids in \cite{RiThKlSk2010}.
\begin{theorem}[Energy conservation]
  Let $\tilde{u}^h$ satisfy equation \eqref{u cg} and $D^h\diff S$ satisfy
  equation \eqref{D cg}. Furthermore assume that $\star\tilde{F}^h$ and
  $\star\tilde{Q}^h$ are defined from \eqref{F eqn}
  and \eqref{Q eqn}.  Then the energy, defined by
\begin{equation}
E = \int_{\Omega} \frac{D^h}{2}\star\tilde{u}^h
\wedge \star (\star\tilde{u}^h) + g\left(
\frac{1}{2}(D^h)^2-b^hD^h\right)\diff S,
\end{equation}
and the potential enstrophy, defined by
\begin{equation}
Z = \int_{\Omega} (q^h)^2D^h\diff S,
\end{equation}
are both conserved. More generally, the energy is conserved for any
$\star\tilde{Q}^h$ satisfying $\star\tilde{Q}^h=\star\tilde{F}^h(q^h)'$ for some
scalar function $(q^h)'$.
\end{theorem}
\begin{proof}
The energy equation is
\begin{eqnarray}
\dot{E}  &=&  \int_{\Omega} D^h(\star\tilde{u}^h)\wedge \star
(\star\tilde{u}^h_t)\diff S + \left(K^h+g(D^h+b^h)\right)\diff S
\wedge\star D^h_t \diff S, \\
&=&  \int_{\Omega} \star\tilde{F}^h \wedge \star
(\star\tilde{u}^h_t) + \star\Pi_2\left(\left(K^h
+ g(D^h+b^h)\right)\diff S\right)\wedge \star D^h_t \diff S,
\end{eqnarray}
where in the second line, we have made use of the definition of
$\star\tilde{F}^h$ taking $\tilde{w}^h=\tilde{u}^h_t$, and we define $\Pi_2$
as the $L_2$ projection into $\hat{\Lambda}_2$, \emph{i.e.}
\begin{equation}
\int_{\Omega} \phi^h\diff S\wedge \star \Pi_2(p\diff S)
= \int_{\Omega}\phi^h\diff S\wedge \star p\diff S,
\quad \forall \phi^h\diff S\in \hat{\Lambda}_2.
\end{equation}
We proceed by substituting equations (\ref{u cg}-\ref{D cg}), with
$\star\tilde{w}^h=\star\tilde{F}^h$ and $\phi^h\diff
S=\Pi_2(K^h + g(D^h+b^h))\diff S)$,
to obtain
\begin{eqnarray}
\nonumber \dot{E} &=&  
-\int_{\Omega} \star\tilde{F}^h \wedge \star\tilde{Q}^h  +
\int_{\Omega} \diff\left(\star\tilde{F}^h\right)
\wedge\star \Pi_2\left((K^h+g(D^h+b^h))\diff S\right) \\
& & \qquad 
 - \int_{\Omega} \diff\left(\star\tilde{F}^h\right)\wedge \star
\Pi_2\left(\left(K^h + g(D^h+b^h)\right)\diff S\right) \\
& = & -\int_{\Omega} q^h\underbrace{\star\tilde{F}^h \wedge \star\tilde{F}^h}_{=0}
=0,
\end{eqnarray}
where in the last line we have made use of the antisymmetry of the 
wedge product.

To show potential enstrophy conservation, we first note that since 
$D^h_t\diff S$ and $\diff(\star\tilde{F}^h)$ are both in
$\hat{\Lambda}^2$, the $L_2$ projection in equation \eqref{D eqn}
is trivial and we obtain
\begin{equation}
\label{D pointwise}
D^h_t\diff S + \diff(\star\tilde{F}^h) = 0,
\end{equation}
pointwise. Now we calculate the enstrophy equation,
\begin{eqnarray}
\dot{Z} & = & \int_{\Omega} q^h\wedge (q^hD^h)_t\diff S + \int_{\Omega}
q^h_t\wedge q^hD^h\diff S,\\ & = & \int_{\Omega} 2q^h\wedge(q^hD^h)_t \diff S -
\int_{\Omega} (q^h)^2\wedge D^h_t \diff S, \\ & = & -\int_{\Omega}
2q^h\wedge\diff(q^h\star\tilde{F}^h) + \int_{\Omega} (q^h)^2 \wedge
\diff(\star\tilde{F}^h), \\ & = & -\int_{\Omega}
(q^h)^2\wedge\diff(\star\tilde{F}^h) - \int_{\Omega} \underbrace{2
  q^h\diff q^h}_{=\diff (q^h)^2} \wedge \star\tilde{F}^h, \\ & = &
-\int_{\Omega} \diff\left((q^h)^2\wedge \star\tilde{F}^h\right) = 0,
\end{eqnarray}
where we have made use of equation \eqref{q cg} using 
$\gamma^h=q^h$, together with 
equation \eqref{D pointwise}, in the third line, and the product
rule and Stokes' theorem (with $\Omega$ closed) in the last line.
\end{proof}
The following property is important for preserving qualitative
properties of $q^h$, since it mimics the Lagrangian conservation and
reduces the types of oscillations that can occur in the solution.
\begin{theorem}[Mass consistent potential vorticity advection]
Let $\tilde{u}^h$ satisfy equation \eqref{u cg} and $D^h\diff S$
satisfy equation \eqref{D cg}, and $\star\tilde{Q}^h$ is defined from
\eqref{Q eqn}, with any choice of $\star\tilde{F}^h$, and suppose that
$D^h>0$ everywhere for all time.

If $q^h$ is initially constant, it will remain constant for all
time.
\end{theorem}
\begin{proof}
Suppose that $q^h$ be constant. For any test function
$(\gamma^h)'\in \hat{\Lambda}^0$, define $\gamma^h\in \hat{\Lambda}^0$
such that 
\begin{equation}
\int_{\Omega} \beta^h \wedge \star (\gamma^h)' = 
\int_{\Omega} \beta^h \wedge \gamma^h D^h\diff S, \quad 
\forall \beta^h\in \hat{\Lambda}^0.
\end{equation}
This is always possible if $D^h>0$. Then
\begin{eqnarray}
\int_{\Omega} (\gamma^h)' \wedge q^h_t\diff S & = & 
\int_{\Omega} \gamma^h \wedge q^h_tD^h\diff S \\
& = & \int_{\Omega} \gamma^h \wedge (q^hD^h)_t\diff S - \int_{\Omega}
\gamma^h \wedge q^hD^h_t\diff S, \\
& = & -\int_{\Omega} \gamma^h \wedge \diff\left(q^h\star\tilde{F}^h\right)
+ \int_{\Omega} \gamma^h \wedge q^h\diff\left(\star\tilde{F}^h\right)\\
& = & -\int_{\Omega} \gamma^h \wedge q^h\diff\left(\star\tilde{F}^h\right)
+ \int_{\Omega} \gamma^h \wedge q^h\diff\left(\star\tilde{F}^h\right) = 0,
\quad \forall (\gamma^h)'\in \hat{\Lambda}^0.
\end{eqnarray}
 where we have made use of equation \eqref{pv cg} together with 
equation \eqref{D pointwise}, in the third line, and have
made use of $q^h$ being constant in the final line. Hence $q^h_t=0$
and $q^h$ remains constant for all time.
\end{proof}
As discussed in \cite{ArHs1990}, the geostrophically balanced
solutions of the rotating shallow-water equations are similar to the
two-dimensional Euler equations in that they exhibit an energy cascade
to large scales, but a potential enstrophy cascade to small scales.
This means that energy conservation is appropriate, but that enstrophy
conservation leads to a pile-up of enstrophy at the gridscale, leading
to very noisy numerical solutions. From a numerical analysis point of
view, we also expect this in our formulation since equation \eqref{pv cg}
 is a continuous finite element Galerkin approximation to the
potential vorticity equation in flux form, with no stabilisation.
This means that it becomes appropriate to introduce terms that
dissipate enstrophy whilst conserving energy and potential vorticity,
and whilst preserving the mass consistency property. We see from the
above that this is possible if we choose
$\star\tilde{Q}^h=(q^h)'\star\tilde{F}^h$ with $(q^h)'=q^h$ if $q^h$ is constant. In
\cite{ArHs1990}, the anticipated potential vorticity method
\cite{SaBa1985} was used as an enstrophy dissipation scheme. Here we
shall show how to introduce this into the finite element framework; we
shall also show that a streamline upwind Petrov-Galerkin (SUPG) scheme
\cite{BrHu1982} can be written in this form and thus conserves energy.
SUPG has the attractive feature of high-order convergence of
solutions.

Furthermore, although balanced solutions have layer thickness $D$
being two derivatives smoother than $q$, and so upwinding $D^h$ is not
always necessary, we are motivated by the use of the shallow-water
equations as a testbed for the horizontal aspects of discretisations
of the three-dimensional Euler equations for numerical weather
prediction, for which the energy also cascades to small scales, and so
we also discuss the use of upwind schemes for $D^h\diff S$, together
with shape preserving limiters, that dissipate potential energy. These
schemes lead to alternate choices of $\star\tilde{F}^h$, and so we can
still have potential vorticity conservation and mass consistency.

Returning to the choice of enstrophy dissipating
$\star\tilde{Q}^h$, the anticipated potential vorticity method
is obtained by setting
\begin{equation}
\label{apvm def}
\star\tilde{Q}^h = \left(q^h-\frac{\tau}{D^h}\contract{\MM{F}^h}\diff q^h\right)
\star\tilde{F}^h = \left(q^h+\frac{\tau}{D^h}\star
(\star\tilde{F}^h\wedge\diff q^h)\right)
\star\tilde{F}^h,
\end{equation}
where $\tau>0$ is an upwind parameter (usually proportional to the
time stepsize $\Delta t$).
\begin{theorem}[Anticipated potential vorticity method 
conserves energy and dissipates enstrophy] Let $\tilde{u}^h$ satisfy
  equation \eqref{u cg} and $D^h\diff S$ satisfy equation \eqref{D
    cg}. Furthermore assume that $\star\tilde{F}^h$,is
  obtained from \eqref{F eqn} and $\star\tilde{Q}^h$ is
  obtained from \eqref{apvm def}. Then energy is conserved and
  enstrophy is dissipated.
\end{theorem}
\begin{proof}
The energy is conserved since
\begin{equation}
  (\star\tilde{Q}^h) \wedge (\star \tilde{F}^h) = 
\left(q^h-\frac{\tau}{D^h}\star(\star\tilde{F}^h\wedge\diff q^h)\right)(\star\tilde{F}^h)
\wedge(\star\tilde{F}^h)=
0,
\end{equation}
so the energy conservation proof is unchanged. The enstrophy
equation becomes
\begin{eqnarray}
\dot{Z} & = & \int 2q^h\wedge (q^hD^h)_t\diff S - \int (q^h)^2D^h_t\diff S \\
 & = & -\int \frac{2\tau}{D^h}\left(\diff q^h\wedge \star\tilde{F}^h\right)
\wedge\star\left(\diff q^h\wedge \star\tilde{F}^h\right)<0.
\end{eqnarray}
\end{proof}
The SUPG flux is given by
\begin{equation}
\label{SUPG def}
\star\tilde{Q}^h = \star\tilde{F}^h\left(q^h 
 - \frac{\tau}{D^h}\star\left(
\pp{q^hD^h}{t}\diff S + \diff(\star\tilde{F}^hq^h)\right)\right),
\end{equation}
where $\tau>0$ is the spatially varying SUPG parameter
given by
\begin{equation}
\tau = \frac{\alpha}{h|\MM{u}^h|},
\end{equation}
with $\alpha>0$ some chosen constant. 
\begin{theorem}[SUPG flux]
Let $\tilde{u}^h$ satisfy
  equation \eqref{u cg} and $D^h\diff S$ satisfy equation \eqref{D
    cg}. Furthermore assume that $\star\tilde{F}^h$,is
  obtained from \eqref{F eqn} and $\star\tilde{Q}^h$ is
  obtained from \eqref{SUPG def}. Then energy is conserved and
  the flux provides an SUPG stabilisation of equation \eqref{pv cg}.
\end{theorem}
\begin{proof}
The energy is conserved since $(\star\tilde{Q}^h) \wedge \star
(\star\tilde{F}^h) = 0$. Equation \eqref{pv cg} becomes
\begin{equation}
\int_{\Omega} \gamma^h \wedge (q^hD^h)_t\diff S 
 = \int_{\Omega} \diff \gamma^h \wedge \star\tilde{F}^hq^h
 - \int_{\Omega} \frac{\tau}{D^h}\diff \gamma^h \wedge \star\tilde{F}^h
\star\left(\pp{q^hD^h}{t}\diff S + \diff(q^h\star\tilde{F}^h)\right),
\end{equation}
and the last term may be rewritten as 
\begin{equation}
\int_{\Omega} \frac{\tau}{D^h}(\star\tilde{F}^h)\wedge\diff\gamma^h \wedge
\star\diff(q^h\star\tilde{F}^h),
\end{equation}
and rearranging gives
\begin{equation}
\int_{\Omega} \left(\gamma^h - \frac{\tau}{D^h}\star(\star\tilde{F}^h\wedge
\diff \gamma^h)\right)\wedge\left(\pp{}{t}(q^hD^h)\diff S + 
\diff(q^h\star\tilde{F}^h)\right)=0,
\end{equation}
which is an SUPG stabilisation of the potential vorticity equation
since $\gamma^h$ has been replaced by $\gamma^h 
 - \frac{\tau}{D^h}\star(\star\tilde{F}^h\wedge
\diff \gamma^h)=\gamma^h + \frac{\tau}{D^h}\MM{F}^h\cdot\nabla q^h$.
\end{proof}

Next we discuss the incorporation of upwinding into the mass equation
\eqref{D eqn}. First we describe the usual upwinding approach, then we
show how an equivalent mass flux $\star\tilde{F}^h$ may be
obtained. Then we discuss slope limiters that are used to enforce
shape preservation when polynomials of degree 1 or greater are used in
$\hat{\Lambda}^2$, and show that an equivalent (time-integrated) mass
flux can be obtained in that case as well.

In deriving an upwind formulation for $D^h$, the integral must be
performed over a single element $e$ due to the discontinuity,
following the discontinuous Galerkin approach.  Taking the $L_2$ inner
product of a test 2-form $\phi^h \diff{S}\in \hat{\Lambda}^2$ with
equation \eqref{D eqn} over one element $e$ gives
\begin{equation}
\dd{}{t}\int_e \phi^h \diff{S} \wedge \star (D^hd S)
= -\int_e\phi^h \diff{S} \wedge \star\diff (D^h\star\tilde{u}^h).
\end{equation}
To obtain coupling between elements we integrate by parts to obtain
\begin{equation}
\label{D dg}
\dd{}{t}\int_e \phi^h \wedge D^h\diff S
= \int_e\diff\phi^h \wedge D^h\star\tilde{u}^h - \int_{\partial e}
\phi^h \wedge D^u\star\tilde{u}^h,
\qquad \forall \phi^h \diff{S} \in \hat{\Lambda}^2,
\end{equation}
where $\partial e$ is the boundary of element $e$ and $D^u$ is
chosen as the value of $D^h$ on the upwind side, following the standard
discontinuous Galerkin approach\footnote{If $k$th order polynomials
  are used, then this flux is $(k+1)$th order accurate. In the case of
  piecewise constant spaces, higher order advection schemes can be
  obtained by reconstructing a higher order upwind flux by
  interpolation from neighbouring elements, using WENO \cite{Sh2009}
  or Crowley schemes \cite{Th1997,Li2005,SkMe2010}, for example.}.  In
contrast with equation \eqref{F eqn}, equation \eqref{D dg} can be
solved locally \emph{i.e.} it only requires the solution of
independent matrix-vector equations in each element to obtain
$\pp{D^h}{t}=0$.
\begin{theorem}[Mass flux for upwind schemes]
\label{mass flux prop}
Let $D^h\diff S\in \hat{\Lambda}^2$ satisfy equation \eqref{D dg}. Then
there exists $\star\tilde{F}^h\in \hat{\Lambda}^1$ such 
that 
\begin{equation}
\label{D ptwise}
\pp{D^h}{t} + \diff\left(\star\tilde{F}^h\right) =0 .
\end{equation}
Furthermore we can calculate $\star\tilde{F}^h$ locally,
\emph{i.e.} independently in each element.
\end{theorem}
\begin{proof}Consider $\star\tilde{F}^h$ constructed from the following
conditions.
\begin{enumerate}
\item \begin{equation}
\int_{\partial e} \phi^h \wedge \star\tilde{F}^h
= \int_{\partial e} \phi^h \wedge D^u\tilde{u}^h,
\qquad \forall \phi^h \diff S\in \hat{\Lambda}^2(e),
\end{equation}
\item \begin{equation}
\int_e \diff\phi^h\wedge \star\tilde{F}^h
= -\int_e \diff\phi^h  \wedge D^h\star\tilde{u}^h, \qquad
\forall \phi^h\diff S\in \hat{\Lambda}^2(e),
\end{equation}
\item \begin{equation}
\int_e \diff \gamma^h \wedge  \star\tilde{F}^h 
= 0, \qquad \forall \gamma^h \in \hat{\Lambda}^0(e) \textrm{ such that }
\gamma^h=0 \textrm{ on }\partial e.
\end{equation}
\end{enumerate}
 This is essentially the Fortin projection into $\hat{\Lambda}^1$
 \cite{Fo1977}. These conditions are unisolvent and lead to the
 correct continuity conditions as shown in \cite{Ar2013}.
 From equation \eqref{D pointwise} we have
\begin{eqnarray}
\dd{}{t}\int_e \phi^h \wedge D^h \diff S & = & -\int_e \phi^h \wedge
\diff(\star\tilde{F}^h) \\
 & = & \int_e \diff \phi^h \wedge \star\tilde{F}^h
 - \int_{\partial e}\phi^h\wedge \star\tilde{F}^h, \\
 & = & \int_e \diff \phi^h \wedge D^h\star\tilde{u}^h
 - \int_{\partial e} \phi^h \wedge D^h\star\tilde{u}^h,
\end{eqnarray}
as required.
\end{proof}
Having obtained $\star\tilde{F}^h$ we can use it in our
definition of $\star\tilde{Q}^h$ and obtain stabilised, conservative,
mass consistent potential vorticity dynamics.

This calculation can be extended to the time-discretised case in which
a slope limiter is applied before each timestep or Runge-Kutta stage
\cite{CoSh2001}. In each element $e$, slope limiters aim to achieve
shape preservation by adjusting $D^h\diff{S}$ in each element $e$ in
such a way that $\bar{D}_e^h$ is preserved, where
\begin{equation}
\bar{D}_e^h = \frac{\int_e D^h \diff{S}}{\int_e \diff{S}}.
\end{equation}
\begin{theorem}[Mass flux from slope limiter]
Let $S(D^h\diff S)$ be the action of a slope limiter on $D^h\diff
S\in\hat{\Lambda}^2$. Then
\begin{equation}
S(D^h\diff{S})= D^h\diff{S}+\diff(\star\tilde{F}^h_s),
\end{equation}
for some slope limiter mass flux $\star\tilde{F}^h_s\in \hat{\Lambda}^1$
that can be calculated locally.
\end{theorem}
\begin{proof}
Consider $\star\tilde{F}^h_s$ constructed from the following
conditions.
\begin{enumerate}
\item \begin{equation}
\int_{\partial e} \phi^h \wedge \star\tilde{F}^h_s
= 0, 
\qquad \forall \phi^h \diff S\in \hat{\Lambda}^2(e),
\end{equation}
\item \begin{equation}
\int_e \diff\phi^h\wedge \star\tilde{F}^h_s
= -\int_e \phi^h  \wedge \left(S(D^h\diff S)-D^h\diff S\right), \qquad
\forall \phi^h\diff S\in \hat{\Lambda}^2(e),
\end{equation}
\item \begin{equation}
\int_e \diff \gamma^h \wedge  \star\tilde{F}^h_s
= 0, \qquad \forall \gamma^h \in \hat{\Lambda}^0(e) \textrm{ such that }
\gamma^h=0 \textrm{ on }\partial e.
\end{equation}
\end{enumerate}
These conditions are again unisolvent. Then
\begin{eqnarray}
\int_e \phi^h \wedge \diff(\star\tilde{F}^h_s)
& = & -\int_e \diff \phi^h \wedge \star\tilde{F}^h_s
 + \int_{\partial e} \phi^h \wedge \star\tilde{F}^h_s, \\
& = & \int_e \phi^h\wedge(S(D^h\diff S)-D^h\diff S), \quad
\forall \phi^h \in \hat{\Lambda}^2.
\end{eqnarray}
Since $S(D^h\diff S)$, $D^h\diff S$ and $\diff\star\tilde{F}^h_s$,
are all elements of $\hat{\Lambda}^2$, this means that 
the projection is trivial and we obtain
\begin{equation}
\diff\left(\star\tilde{F}^h_s\right) = S(D^h\diff S)-D^h\diff S,
\end{equation}
as required.
\end{proof}

\subsection{Finite element discretisation: primal-dual grid finite
  element formulation}
\label{dual grid formulation}

In this section we provide an alternative formulation that makes use
of a second set of spaces defined on a dual grid based on the second
vector field proxy. The introduction of the dual grid means that we
can now express $\nabla$ and $\nabla^\perp\cdot$ in a strong
form, in addition to $\nabla^\perp$ and $\nabla\cdot$.  The idea is
that when we want to apply $\nabla^\perp$ and $\nabla\cdot$ operators
strongly we use the primal grid spaces as defined in the previous
section, and when we want to apply $\nabla$ and $\nabla^\perp\cdot$
strongly we use the dual grid spaces. This requires defining mappings
between the primal and dual spaces which are defined \emph{via} the
Hodge star operator. We shall observe that primal dual and primal-dual
formulations are exactly equivalent for the linear equations, but that
they differ for the nonlinear equations; the primal-dual scheme may
facilitate some alternative handling of nonlinear terms that gives
some advantage. One particular benefit is that locally conservative
discontinuous schemes can then be used for both mass and potential
vorticity.

We start by selecting a set of finite element differential form spaces
$\hat{\Lambda}^k_p\subset \Lambda^k$ and $\hat{\Lambda}^k_d\subset
\Lambda^k$ on the primal grid and the dual grid respectively,
satisfying
\begin{equation}
 \begin{CD}
  \hat{\Lambda}^0_p @>\diff >> \hat{\Lambda}^1_p @>\diff >> 
  \hat{\Lambda}^2_p \\
  \hat{\Lambda}^0_d @>\diff >> \hat{\Lambda}^1_d @>\diff >> 
  \hat{\Lambda}^2_d.
 \end{CD}
\end{equation}
We shall calculate with mass $D^h\diff{S}$ in $\hat{\Lambda}^2_p$, and
vorticity $\zeta^h\diff{S}$ in $\hat{\Lambda}^2_d$, which are both
discontinuous finite element spaces where locally conservative
discontinuous methods can be used. We shall use the flux 1-form
representation of velocity $\star\tilde{u}^h$ in $\hat{\Lambda}^1_p$ (to
evaluate divergence) but also work with a consistent circulation
1-form representation of velocity $\tilde{v}^h$ in $\hat{\Lambda}^1_d$
(to evaluate vorticity), with $\tilde{u}^h\ne\tilde{v}^h$ but related
through an appropriate mapping.

\paragraph{Discrete Hodge star}
Following \cite{Hi2001a,Hi2001b}, we define discrete Hodge star operators:
$\star_h:\hat{\Lambda}_d^k\to\hat{\Lambda}_p^{2-k}$ given by
\begin{equation}
\int \gamma^h \wedge \star (\star_h\omega^h) = (-1)^k\int \gamma^h \wedge
\omega^h, \quad \forall \gamma^h \in \hat{\Lambda}^k_p,
\end{equation}
and require that the dual spaces are chosen such that $\star_h$ is
invertible. This requirement somewhat limits the choice of spaces.
\cite{BuCh2007} (see also \cite{Ch2008}) proved that $\star_h$ is
invertible for the hexagonal P1-RT0-P0 spaces for the primal mesh, and
the triangular P1-N0-P0 spaces for the dual mesh, which would have the
same degree-of-freedom points as the hexagonal C-grid. We have also
observed numerically that $\star_h$ is invertible for similar spaces with
quadrilaterals for the primal mesh.

The discrete dual operator $\delta^h$ provides weak approximations to
$\nabla$ and $\nabla^\perp\cdot$ in the primal space, as described in
the previous section for the primal finite element formulation, and
weak approximations to $\nabla^\perp$ and $\nabla\cdot$ in the dual
space. We shall now see that the key to the formulation is that we can
define a simple relationship \emph{via} the discrete Hodge star
$\star_h$ between $\diff$ in the primal space and $\delta^h$ in the dual
space, and \emph{vice versa}.

\begin{theorem}[Mapping from $\diff$ to $\delta^h$]
For $\omega^h\in\hat{\Lambda}^k_d$, $k=0,1$,
$\star_h\diff\omega^h=\delta^h\star_h\omega^h$, and hence the
following diagram commutes:
\begin{equation}
 \begin{CD}
  @>\delta^h >>  @>\delta^h >> \\
 \hat{\Lambda}^2_d @<\diff << \hat{\Lambda}^1_d @<\diff << \hat{\Lambda}^0_d \\ 
 @VV{\star_h}V @VV{\star_h}V @VV{\star_h}V \\
 \hat{\Lambda}^0_p @> d >> \hat{\Lambda}^1_p @>\diff >> \hat{\Lambda}^2_p \\
  @<\delta^h <<  @< \delta^h <<
 \end{CD}
\end{equation}
\end{theorem}
\begin{proof}
\begin{eqnarray}
\int \gamma^h \wedge \star (\star_h \diff \omega^h)
&=&(-1)^{k-1}\int \gamma^h \wedge \diff \omega^h \\
&=&(-1)^k\int \diff \gamma^h \wedge \omega^h \\
&=&(-1)^k\int \diff \gamma^h \wedge \star (\star_h \omega^h) \\
&=& \int \gamma^h \wedge \star \delta^h (\star_h \omega^h), \quad
\forall \gamma^h \in \hat{\Lambda}^{k-1}_p,
\end{eqnarray}
where we may integrate by parts in the second step since $\diff\gamma^h$
and $\diff\omega^h$ are both well-defined. 
\end{proof}
For example, this means that we may start with
$\star\tilde{u}^h\in \hat{\Lambda}^1_p$, and either obtain
the primal vorticity $\zeta^h_p\in \hat{\Lambda}^0_p$ by directly
applying $\delta^h$, or by inverting $\star_h$ to get
$\tilde{v}^h \in \hat{\Lambda}^1_d$, applying
$\diff$ to get the dual vorticity $\zeta^h_d\diff{S}\in
\hat{\Lambda}^2_d$, and then projecting back to $\hat{\Lambda}^0_p$ with $\star_h$,
 \emph{i.e.},
\begin{equation}
\delta^h(\star\tilde{u}^h)=\star_h\diff\star_h^{-1}(\star\tilde{u}^h).
\end{equation}

Next we introduce the primal-dual grid version of equations (\ref{u
  eqn}-\ref{D eqn}). Within this framework, we retain the same
equation for $D^h$ on the primal grid, and modify the Coriolis term in
the velocity equation as follows:
\begin{eqnarray}
\label{u eqn dual}
\dd{}{t}\int_{\Omega}(\star\tilde{w}^h)\wedge\star(\star\tilde{u}^h) + \int_{\Omega}\star\tilde{w}^h\wedge\star_h \tilde{Q}^h 
- \int_{\Omega}\diff(\star\tilde{w}^h)\wedge\left(g(D^h+b^h)+K^h\right)
, \qquad \forall \star\tilde{w}^h  \in \hat{\Lambda}^1_p,
\end{eqnarray}
with $\tilde{Q}^h\in \hat{\Lambda}^1_d$, and
$\star\tilde{F}^h\in \hat{\Lambda}^1_p$. Using
$\delta^h$ and $\star_h$, we can rewrite the velocity equation as 
\begin{equation}
\label{u dp}
\pp{}{t}\star\tilde{u}^h + \star_h\tilde{Q}^h
+ \delta^h\left(g(D^h+b^h)+K^h\right)=0.
\end{equation}
\begin{theorem}[Primal vorticity conservation]
  Let $\star\tilde{u}^h\in \hat{\Lambda}^1_p$ and $D^h\diff S\in
  \hat{\Lambda}^2_p$ satisfy Equations \eqref{u eqn dual} and \eqref{D
    ptwise} respectively. Then the primal vorticity defined by $\zeta^h_p = \delta^h
  \star\tilde{u}^h$ satisfies
\begin{equation}
\label{zeta p}
\pp{}{t}\zeta^h_p + \delta^h\star_h\tilde{Q}^h
=0.
\end{equation}
Furthermore, the primal vorticity is conserved.
\end{theorem}
\begin{proof}
Applying $\delta^h$ to equation \eqref{u dp} gives
\begin{equation}
\pp{}{t}\delta^h\star\tilde{u}^h + \delta^h\star_h\tilde{Q}^h
=0,
\end{equation}
since $(\delta^h)^2=0$. Global conservation follows directly since
\begin{eqnarray}
\int_{\Omega} \zeta^h_p \diff S = \int_{\Omega} \underbrace{(\diff 1)}_{=0}
\wedge \tilde{Q}^h=0.
\end{eqnarray}
\end{proof}
\begin{theorem}[Dual vorticity conservation]
Let $\star\tilde{u}^h\in \hat{\Lambda}^1_p$ and $D^h\diff S\in
\hat{\Lambda}^2_p$ satisfy Equations \eqref{u eqn dual} and \eqref{D
    ptwise} respectively. Then the dual grid vorticity $\zeta^h_d\diff{S}\in
\hat{\Lambda}^2_0$ given by $\star_h\zeta^h_d\diff{S}=\zeta^h_p$
satisfies
\begin{equation}
\label{zeta_d eqn}
\pp{}{t}\zeta^h_d\diff S + \diff\tilde{Q}^h=0.
\end{equation}
Furthermore, $\zeta^h_d$ is locally conserved.
\end{theorem}
\begin{proof}
Substituting
$\star\tilde{u}^h=\star_h\tilde{v}^h$ with
$\tilde{v}^h\in \hat{\Lambda}^1_d$, we obtain
\begin{equation}
\star_h\zeta^h_d\diff{S} = \zeta^h_p = \delta^h\star\tilde{u}^h
= \delta^h\star_h\tilde{v}^h=\star_h\diff\tilde{v}^h,
\end{equation}
and hence $\zeta^h_d=\diff\tilde{v}^h$ by
invertibility of $\star_h$. Substitution into equation \eqref{zeta p}
and application of the commutation relations for $\star_h$ gives
\begin{equation}
\pp{}{t}\star_h\zeta^h_d\diff S + \star_h\diff\tilde{Q}^h
=0,
\end{equation}
and hence we obtain equation \eqref{zeta_d eqn} by invertibility of
$\star_h$. Local conservation follows since
\begin{equation}
\dd{}{t}\int_{e'} \zeta^h_d\diff S 
= -\int_{e'}\diff\tilde{Q}^h=-\int_{\partial e'}\tilde{Q}^h,
\end{equation}
for each dual element $e'$, and local conservation follows from
the appropriate continuity of $\tilde{Q}^h$, 
so that  the flux integral takes the same value on either side
of $\partial e'$.
\end{proof}
We now make a particular choice of $\tilde{Q}^h$, guided
by the requirement of mass consistent advection of the dual potential
vorticity $q_d^h\diff S\in \hat{\Lambda}^2_d$, defined by
\begin{equation}
\label{pv def dual}
\int \phi^h \wedge q_d^hD^h\diff{S}
= \int \phi^h \wedge 
(\zeta^h_d+f)\diff{S}, \qquad \forall \phi^h\diff{S}\in \hat{\Lambda}^2_d.
\end{equation}
Here we are seeking locally conservative schemes that dissipate
potential enstrophy, and hence we propose the following upwind scheme,
\begin{equation}
\label{pv eqn dual}
\dd{}{t}\int_{e_d}\phi^h \wedge q^h_d D^h\diff{S} - \int_{e_d}\diff \phi^h\wedge
q^h_d\star\tilde{F}^h + \int_{\partial e_d}
\phi^h \wedge q^u_d \star\tilde{F}^h\diff{S} = 0,
\quad \forall \phi^h\diff S \in \hat{\Lambda}^2_d,
\end{equation}
for each dual element $e_d$ with boundary $\partial e_d$, where
$q^u_d$ is an appropriate upwind flux\footnote{Here the
  only known cases of invertible $\star_h$ are with piecewise constant
  $\hat{\Lambda}^2_d$ spaces, and so reconstruction is required
  to obtain higher order fluxes.} that takes the same
values on both sides of the boundary $\partial e_d$.
\begin{theorem}[Dual potential vorticity conservation and mass consistency]
  Let $\star\tilde{u}^h\in \hat{\Lambda}^1_p$ and $D^h\diff S\in
  \hat{\Lambda}^2_p$ satisfy Equations \eqref{u eqn dual} and \eqref{D
    ptwise} respectively. There exists a choice of $\tilde{Q}^h\in
  \hat{\Lambda}^1_d$ such that the diagnosed dual potential vorticity
  $q_d^h\diff S\in \hat{\Lambda}^2_d$ obtained from equation \eqref{pv
    def dual} satisfies equation \eqref{pv eqn dual}, and
  $\tilde{Q}^h$ can be obtained locally in each element.
  Hence $q^h_d\diff S$ is locally conserved.  Furthermore, if $q^h_d$ is
  is constant then $\pp{q^h_d}{t}=0$.
\end{theorem}
\begin{proof}
  Define $\tilde{Q}^h$ on each dual element $e_d$ according to the
  following conditions:
\begin{enumerate}
\item
\begin{equation}
\int_{f_d}\phi^h \wedge \tilde{Q}^h  =  
\int_{f_d}\phi^h\wedge q^u_d\star\tilde{F}^h, \quad
\forall \phi^h\diff S \in \hat{\Lambda}^2_d(e_d), 
\end{equation}
where $f_d$ are the edges of $e_d$,
\item 
\begin{equation}
\int_{e_d} \diff\phi^h\wedge\tilde{Q}^h =\int_{e_d}\diff\phi^h\wedge
q^h_d\star\tilde{F}^h 
\quad
\forall \phi^h\diff S \in \hat{\Lambda}^2_d(e_d), 
\end{equation}
\item 
\begin{equation}
\int_{e_d}\diff \gamma^h \wedge \tilde{Q}^h = 0, \qquad
\forall \gamma^h \in \hat{\Lambda}^0_d(e_d) \textrm{ with }
\gamma^h  = 0 \textrm{ on } \partial e_d.
\end{equation}
\end{enumerate}
These conditions are sufficient to determine $\tilde{Q}^h$
when restricted locally to $e_d$, and lead to the correct continuity
conditions. Then
\begin{eqnarray}
\dd{}{t}\int_{e_d} \phi^h \wedge q^h_dD^h\diff S & = & 
-\int_{e_d} \phi^h\wedge \diff \tilde{Q}^h \\
& = & \int_{e_d} \diff \phi^h \wedge \tilde{Q}^h
- \int_{\partial e_d} \phi^h \wedge \tilde{Q}^h \\
& = & \int_{e_d} \diff \phi^h \wedge q^h\star\tilde{F}^h
- \int_{\partial e_d} \phi^h \wedge q^u_d\star\tilde{F}^h
\end{eqnarray}
as required.  Local conservation follows from choosing $\phi^h$ constant
in $e_d$, giving
\begin{equation}
\dd{}{t}\int_{e_d}q^h_dD^h\diff S = -\int_{\partial e_d}q^u_d\star\tilde{F}^h,
\end{equation}
and appropriate continuity of $\star\tilde{F}^h$ and
$q^u_d$.  To show mass consistency, suppose that $q^h_d$ be
constant. For any test function $(\phi^h)'\diff S\in \hat{\Lambda}^2_d$,
and a given dual element $e_d$, define $\phi^h\diff S\in
\hat{\Lambda}^2$ such that
\begin{equation}
\int_{e_d} \beta^h \diff S \wedge \star (\phi^h)'\diff S = 
\int_{e_d} \beta^h \diff S \wedge \star \phi^h D^h\diff S, 
\quad \forall \beta^h\diff S \in 
\hat{\Lambda}^2_d.
\end{equation}
This is always possible when $D^h>0$. 
Then from equation 
\eqref{pv eqn dual},
\begin{eqnarray}
\dd{}{t}\int_{e_d} (\phi^h)' \wedge \pp{q^h_d}{t}\diff S & = & 
\int_{e_d} \phi^h \wedge \pp{q^h_d}{t}D^h\diff S \\
& = & \int_{e_d} \phi^h \wedge (q^h_dD^h)_t\diff S - \int_{e_d} \phi^h
\wedge q^h_dD^h_t\diff S \\
&=&  \int_{e_d}\diff \phi^h\wedge q^h_d\star\tilde{F}^h
- \int_{\partial e_d}
\phi^h \wedge \underbrace{q^u_d}_{=q^h_d} \star\tilde{F}^h \\
& & \qquad + \int_{e_d} \phi^h \wedge q^h_d\diff(\star\tilde{F}^h)=0,
\quad \forall (\phi^h)'\diff S \in \hat{\Lambda}^2_d,
\end{eqnarray}
 where in the last line we have used the fact the $q^h_d$ is 
constant so $q^u_d=q^h_d$, together with Stokes' theorem.
Hence $\pp{q^h_d}{t}=0$.
\end{proof}

\section{Numerical tests}
\label{numerical}
In this section we provide some numerical results, primarily to show
that these are not just theoretical results of academic interest but
can be used in practical codes (with the potential to extend to
3D). Extensive test case results quantifying accuracy and convergence,
as well as demonstrating the desirable properties for which the
schemes are designed, will be presented in subsequent publications.
In particular, we benchmark using our schemes on a standard
meteorological test case, namely case 5 from \citep{Wi1992}, which is
a steady balanced flow on the sphere which is disturbed at time $t=0$
by the appearance of a conical mountain at mid-latitudes.  The errors
are computed by comparing the layer depth at 15 days with a resolved
pseudo-spectral solution as prescribed in the test case specification.

The test was run with four different finite element spaces on
triangles with an icosahedral mesh using the primal formulation, and
two different finite element spaces using the primal-dual formulation,
one on hexagons with a dual icosahedral mesh and one on quadrilaterals
on a cube mesh. The primal spaces used were P1/RT0/P0 (Linear
continuous for vorticity, lowest Raviart-Thomas space for velocity,
piecewise constant for pressure, denoted the
$\mathcal{P}_0^-\Lambda^k$ spaces in \cite{ArFaWi2006}), P2/BDM1/P0
(Quadratic continuous for vorticity, lowest Brezzi-Douglas-Marini
space for velocity, piecewise constant for pressure, denoted the
$\mathcal{P}_{2-k}\Lambda^k$ spaces in \cite{ArFaWi2006}),
P3/BDM2/P1DG (Cubic continuous for vorticity, second
Brezzi-Douglas-Marini space for velocity, discontinuous linear for
pressure, denoted the $\mathcal{P}_{3-k}\Lambda^k$ spaces in
\cite{ArFaWi2006}), and P2B/BDFM1/P1DG (Quadratic continuous with
cubic bubbles for vorticity, Brezzi-Douglas-Fortin-Marini space for
velocity, discontinuous linear for pressure, as discussed in
\cite{CoSh2012}). The energy-conserving scheme was used for layer
depth and the APVM stabilisation was used for potential vorticity,
with centred-in-time semi-implicit time integration. The primal-dual
spaces were lowest order P1/RT1/P0 on the dual mesh of the icosahedral
triangulation using the construction of \cite{BuCh2007,Ch2008}, and an
analogous construction on a cubed sphere mesh made of quadrilaterals.
Third-order space-time upwind schemes were used for both layer depth
and potential vorticity by using a Crowley scheme to interpolate the
high-order flux from neighbouring elements. Error plots are 
shown in Figure \ref{error table}, and approximately second-order
convergence is observed for all schemes except for BDM1.

\begin{figure}
\centerline{
\includegraphics[width=10cm]{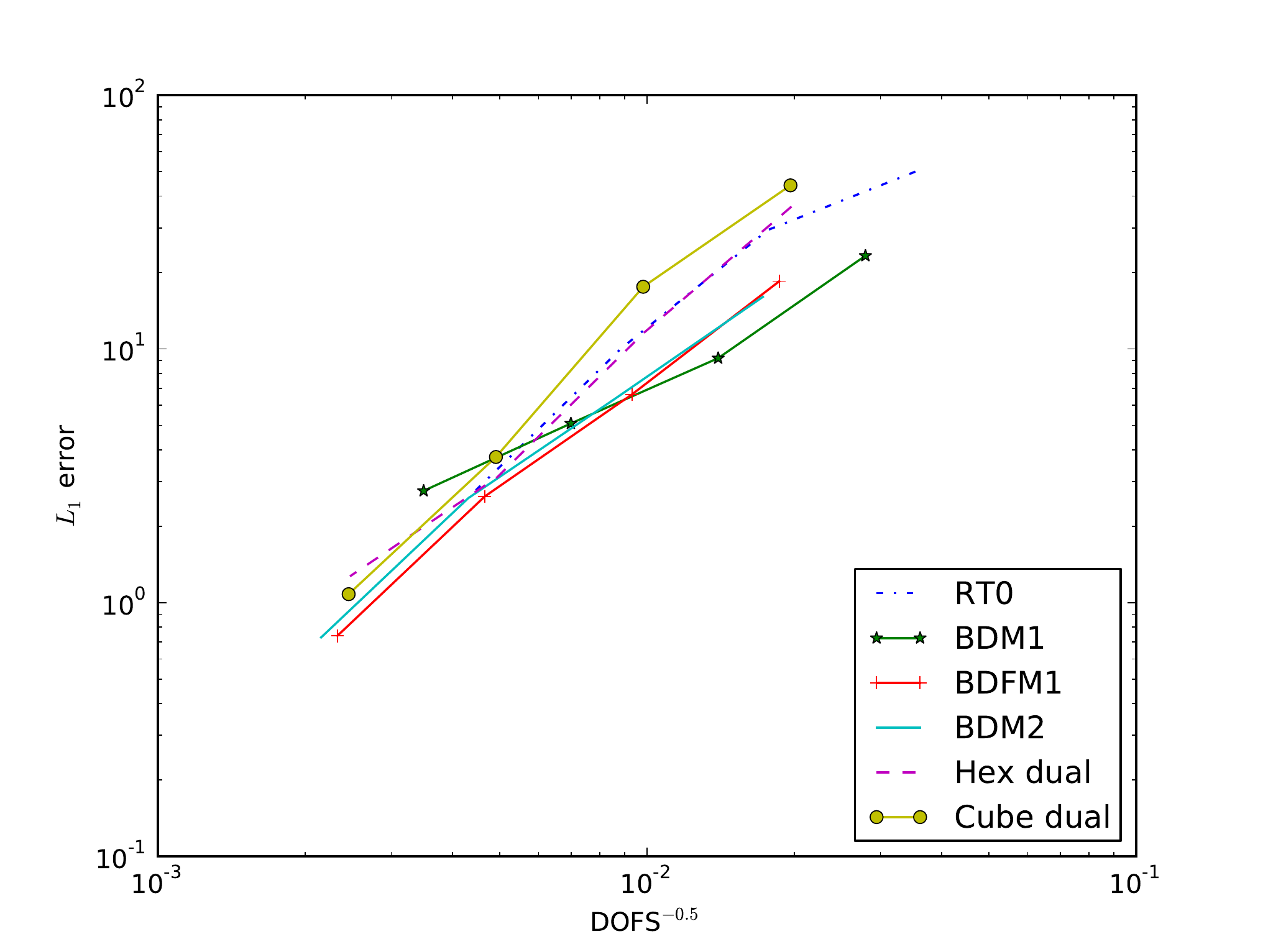}
\includegraphics[width=10cm]{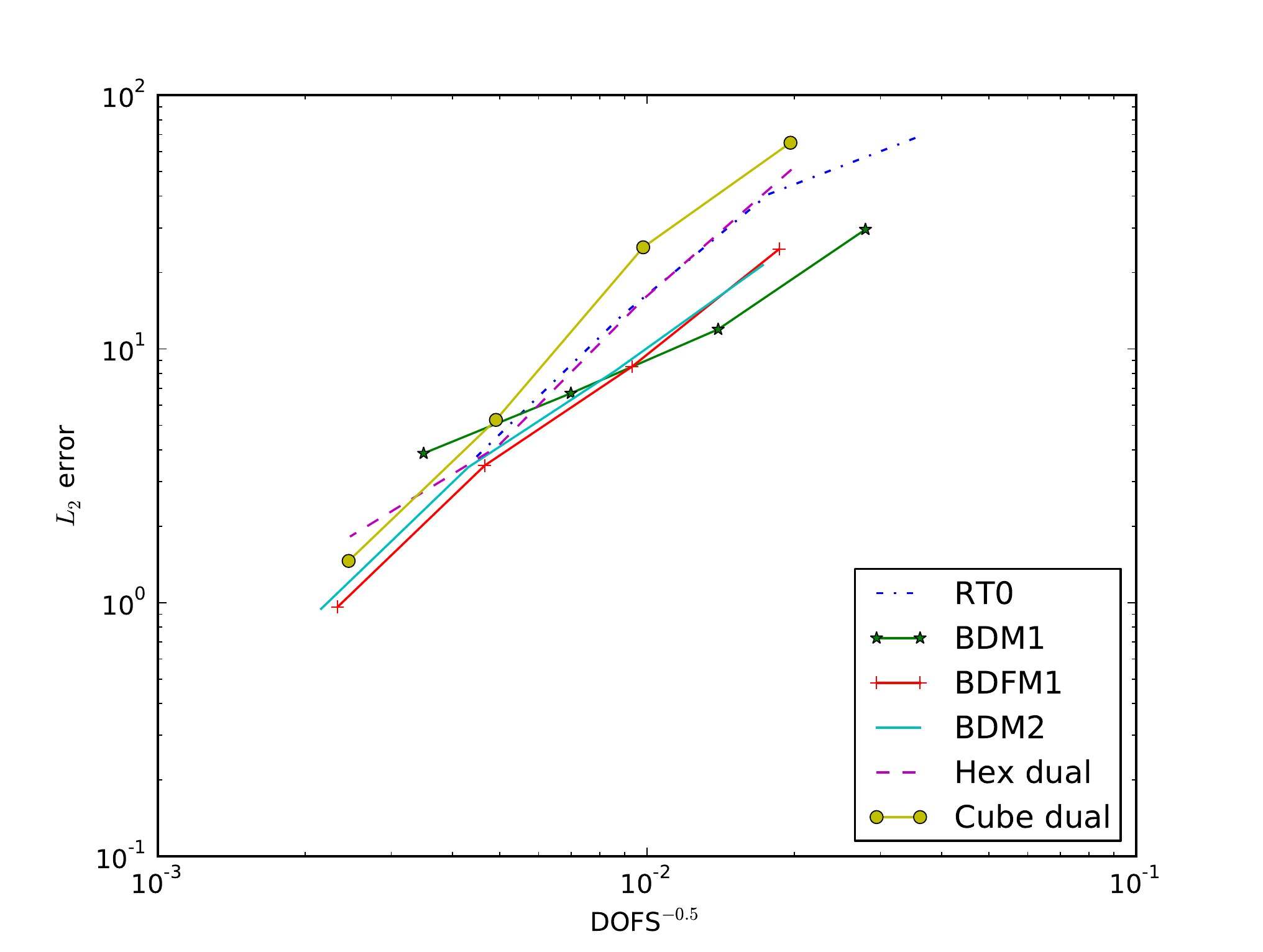}}
\centerline{\includegraphics[width=10cm]{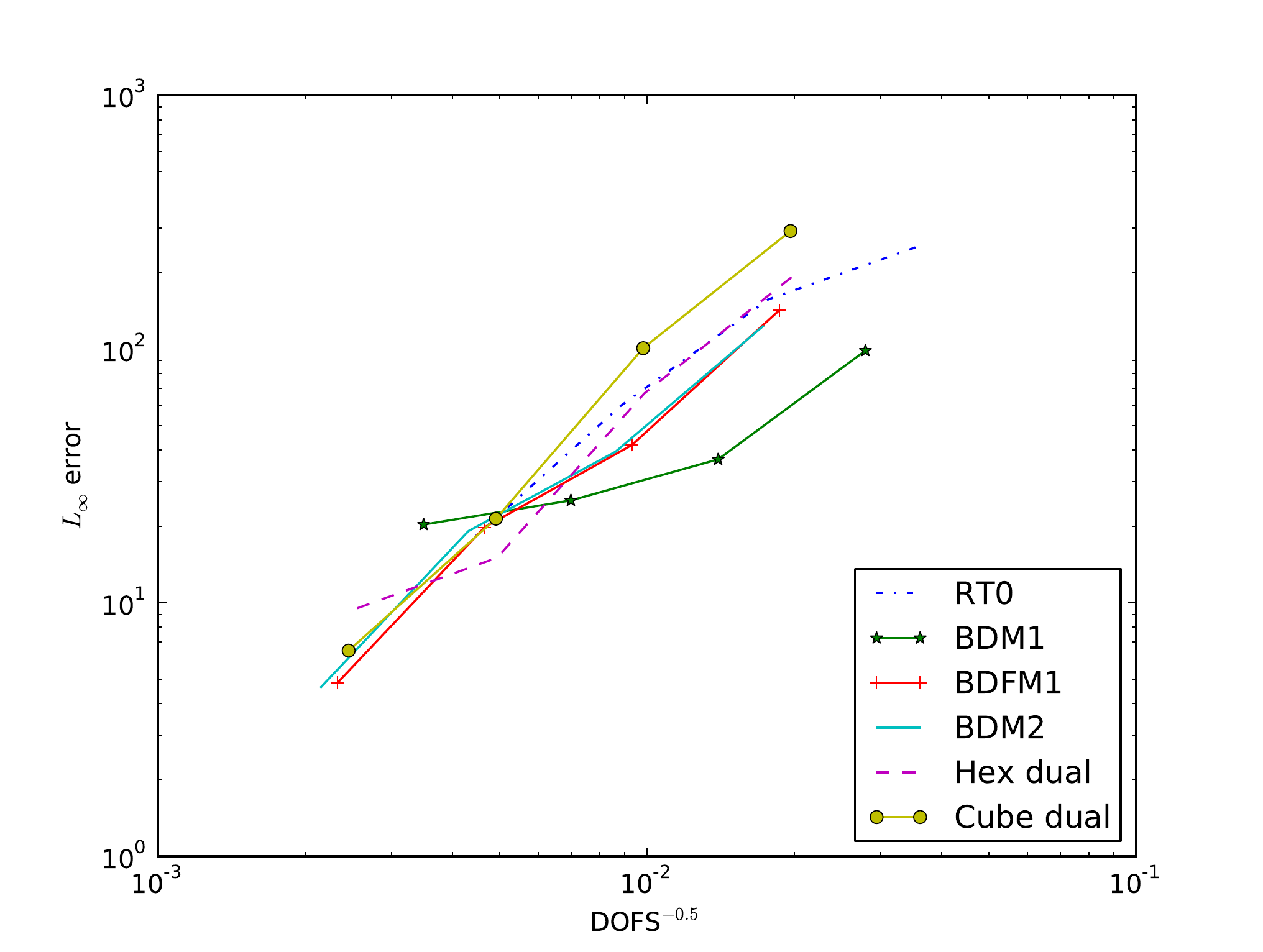}
}
\caption{\label{error table}Tables showing numerical errors 
  in the layer depth at 15 days for the standard ``flow over a mountain''
  testcase, relative to a resolved pseudo-spectral solution, for
  various different finite element spaces.}
\end{figure}

\section{Summary and outlook}
\label{summary and outlook}
In this paper, we used the finite element exterior calculus framework
to develop two formulations for the shallow-water equations, a primal
formulation that is defined on a single mesh where divergence is
defined in the strong form but vorticity must be evaluated weakly
using integration by parts, and a primal-dual formulation that makes
use of a dual mesh where vorticity can also be computed in the strong
form. Both of these formulations have a conserved diagnostic potential
vorticity, that satisfies mass consistency \emph{i.e.} constant $q$
stays constant. In the primal mesh case we are able to choose energy
and enstrophy conserving mass and potential vorticity fluxes.  Both
formulations provide a way to control oscillations in the
divergence-free component of the velocity field (the component that
dominates in large scale balanced flow in the atmosphere) by ensuring
that the potential vorticity remains mass-consistent and
oscillation-free. In the primal-dual case this can be achieved since
the potential vorticity is diagnosed on a discontinuous space where
discontinuous Galerkin/finite volume methods can be used to provide
stable shape-preserving potential vorticity fluxes. In the primal
case, the potential vorticity is computed in a continuous finite
element space, but streamline-upwind Petrov-Galerkin methods with
discontinuity capturing are compatible with the framework and can be
used to provide stable potential vorticity fluxes.

This work is part of the UK GungHo Dynamical Core project, which is a
NERC/STFC collaboration between UK academics and the UK Met Office to
design a dynamical core for the Unified Model that will perform well
on the next generation of massively parallel supercomputers. In Phase
1 of the project, one of the main goals is to determine the horizontal
discretisation that will be used, with the shallow-water equations on
the sphere providing an environment to investigate this. The aim is to
develop discretisations on a pseudouniform grid\footnote{A grid for
  which the ratio of smallest to largest edge lengths remains bounded
  as the maximum edge length tends to zero.} that have all of the
desirable properties listed in the introduction, whilst maintaining
the accuracy of the current model. Numerical accuracy is crucial since
it reduces grid imprinting (structure in the numerical errors that
reflects the structure of the grid, \emph{e.g.} larger errors near the
corners of a cubed sphere). This work opens up a number of
possibilities that could be sufficiently accurate for operational
use. In \cite{CoSh2012} it was shown that to avoid spurious mode
branches it is necessary to select finite element spaces that have a
2:1 ratio of velocity DOFs to pressure DOFs, which suggests the BDFM1
space on triangles with an icosahedral mesh in the primal formulation
or RT0 on quadrilaterals with a cubed sphere mesh in the primal-dual
formulation. There is an argument to be made that spurious Rossby mode
branches arising from increasing velocity DOFs relative to this ratio
are not harmful since they have very low frequencies and will just be
passively advected by the flow. This suggests the BDM1 space on
triangles in the primal formulation or the RT0 space on hexagons in
the primal-dual formulation, which both have a 3:1 ratio. The next
steps in this work are to analyse the numerical convergence and
dispersion relations of all of these schemes and to benchmark them
against the usual suites of testcases and against solutions from the
Unified Model formulation.

\paragraph{Acknowledgements} The authors would like to thank Thomas
Dubos for suggesting to look at a primal-dual finite element
formulation, Darryl Holm for useful guidance and comments on the
paper, Andrew McRae for providing primal scheme testcase results,
Hilary Weller for providing reference pseudospectral solutions, and
the GungHo UK dynamical core team for interesting discussions and
debate. This work is supported by the Natural Environment Research
Council.

\bibliographystyle{model1-num-names}
\bibliography{feec_sw}







\end{document}